\documentclass{article}

\usepackage{amsmath,amsthm,amsfonts,graphicx}
\usepackage{multirow}
\usepackage{slashbox}
\usepackage{multirow}
\def\smallddots{\mathinner{\raise7pt\hbox{.}\raise4pt\hbox{.}\raise1pt\hbox{.}}}
\def\smallsdots{\mathinner{\raise1pt\hbox{.}\raise4pt\hbox{.}\raise7pt\hbox{.}}}

\DeclareMathOperator{\diag}{diag}

\DeclareMathOperator{\rank}{rank}
\DeclareMathOperator{\nrank}{nrank}

\newtheorem{theorem}{Theorem}[section]
\numberwithin{equation}{section}
\numberwithin{table}{section}
\newtheorem{lemma}{Lemma}[section]

\newtheorem{corollary}{Corollary}[section]

\newtheorem{algorithm}{Algorithm}[section]

\newtheorem{definition}{Definition}[section]

\newtheorem{remark}{Remark}[section]

\newtheorem{problem}{Problem}[section]
\begin{document}

\title{Fast Derandomized Low-rank Approximation and  
 Extensions
\thanks {The results of this paper have been  
presented at 
the Eleventh International Computer Science Symposium in Russia 
(CSR'2016),
in St. Petersbourg, Russia, June 2016  (see \cite{PZ16}).}} 
\author{Victor Y. Pan} 

\author{Victor Y. Pan$^{[1, 2],[a]}$, Liang Zhao$^{[2],[b]}$, and
John Svadlenka$^{[2],[c]}$
\\
\and\\
$^{[1]}$ Department of Mathematics and Computer Science \\
Lehman College of the City University of New York \\
Bronx, NY 10468 USA \\
$^{[2]}$ Ph.D. Programs in Mathematics  and Computer Science \\
The Graduate Center of the City University of New York \\
New York, NY 10036 USA \\
$^{[a]}$ victor.pan@lehman.cuny.edu \\
http://comet.lehman.cuny.edu/vpan/  \\
$^{[b]}$  lzhao1@gc.cuny.edu \\
$^{[c]}$ jsvadlenka@gradcenter.cuny.edu 
} 
\date{}

\maketitle


\begin{abstract} 

\begin{itemize}
  \item
\noindent Low-rank approximation of a matrix by means of structured random sampling 
has been consistently efficient in its extensive empirical studies 
around the globe, but adequate formal support for this empirical phenomenon 
has been missing  so far.
\item
Based on our novel insight into the subject, we provide
such an elusive formal support and
derandomize and simplify
the known numerical algorithms
for low-rank approximation and related computations. 
\item
Our techniques can be applied to some other 
areas of fundamental matrix computations, in particular 
to the Least Squares Regression, 
 Gaussian elimination with no pivoting and block  Gaussian elimination.
\item
Our formal results and our numerical tests
 are in good accordance with each other.
\end{itemize}
\end{abstract}




\section{Introduction}\label{sintr}


\subsection{The problem of low-rank approximation and our progress briefly}\label{sprbpr}


{\em Low-rank approximation of a 
matrix}  has a variety of applications to
the most
fundamental matrix computations  \cite{HMT11} and
 numerous problems of data mining and analysis,
``ranging from term document data to DNA SNP data" \cite{M11}.
Classical solution algorithms  
use SVD or rank-re\-veal\-ing
factorizations, but
the alternative solution 
by means of random sampling
is numerically reliable, robust, and
computationally and conceptually simple
and has become highly and increasingly popular
 in the last decade
(see  \cite{HMT11}, \cite{M11}, and
\cite[Section 10.4.5]{GL13}
for surveys and ample bibliography).

In particular the paper \cite{HMT11} proves that
random sampling  algorithms applied with Gaussian multipliers
 produce low-rank approximation
with a probability close to 1,
but empirically the algorithms work as efficiently
 with various random structured multipliers.
  
Adequate formal support for this empirical evidence has been elusive so far, 
but based on our new insight we obtain such a support and furthermore
 derandomize these algorithms and  simplify them
by applying them with some sparse and structured multipliers.  
The known links 
enable immediate extensions of our results to
 various important computations in numerical linear algebra 
and data mining and analysis, but we also
extend them to other fundamental computational problems 
solved by using random multipliers.
We outline our results in this section. They are in good accordance with
 our numerical tests
of Section \ref{ststs}.


\subsection{Some definitions}\label{ssdef}


\begin{itemize}
  \item
Typically we use the concepts ``large", ``small", ``near", ``close", ``approximate", 
``ill-conditioned" and ``well-conditioned"  
quantified in the context, but we specify them quantitatively as needed. 
\item
Hereafter 
``$\ll$" means ``much less than";
{\em ``flop"} stands for ``floating point arithmetic operation".
\item
$I_s$ is the $s\times s$ identity matrix.  $O_{k,l}$ is a $k\times l$
matrix filled with zeros. ${\bf o}$ is a vector filled with zeros.
\item
 $(B_1~|~B_2~|~\dots~|~B_h)$ 
denotes a $1\times h$ block matrix with the blocks $B_1,B_2,\dots,B_h$.
\item
$\diag(B_1,B_2,\dots,B_h)$
denotes a $h\times h$ block diagonal matrix with  
diagonal  blocks $B_1,B_2,\dots,B_h$.
\item
 $\rank(W)$, $\nrank(W)$, and $||W||$ denote the  {\em rank}, 
 {\em numerical rank}, and the {\em spectral norm} of a matrix $W$, respectively. 
\item
$W^T$ and $W^H$, and denote its   
 transpose and Hermitian  transpose, respectively.
\item
An $m\times n$ matrix $W$ is called  {\em unitary}  
if $W^HW=I_n$ or if  $WW^H=I_m$. If this matrix is known to be real, then it is also
and preferably  called {\em orthogonal}. 

($||UW||=||W||$ and $||WU||=||W||$ if the matrix $U$ is unitary.)
\item
$W=S_{W,\rho}\Sigma_{W,\rho}T^T_{W,\rho}$ is 
 {\em  compact SVD}
of a matrix $W$ of rank
$\rho$ with
 $S_{W,\rho}$ and $T_{W,\rho}$ denoting the  unitary
matrices of its singular vectors and
$\Sigma_{W,\rho}=\diag(\sigma_j(W))_{j=1}^{\rho}$ the
diagonal matrix of its singular values
in non-increasing order,
$\sigma_1(W)\ge \sigma_2(W)\ge \dots\ge 
\sigma_{\rho}(W)>0$. ($\sigma_1(W)=||W||.$)
\item
 $\kappa(W)=\sigma_1(W)/\sigma_{\rho}(W)\ge 1$
 denotes the  {\em condition number}
of a matrix $W$. 
A matrix
 is called
{\em ill-conditioned} if 
its condition number 
 is large in context and
is called {\em well-conditioned}
if  this number
$\kappa(W)$ is reasonably bounded.

(An $m\times n$ matrix is ill-conditioned
if and only if it has a matrix of a smaller rank nearby
or equivalently
 if and only if its rank exceeds 
its numerical rank; 
an $m\times n$ matrix is well-conditioned 
if and only if it has full numerical rank $\min\{m,n\}$.
A matrix $W$ is unitary if and only if $\kappa(W)=1$.)

 \item
{\em ``Likely"}  
means ``with a probability close to 1", 
the acronym ``{\em i.i.d.}" stands
for
``independent identically distributed",
and we
refer to ``standard Gaussian random" variables 
  just as ``{\em Gaussian}". 
\item
We call an $m\times n$ matrix {\em  Gaussian} and denote it $G_{m,n}$ 
if all its entries are  i.i.d. Gaussian variables. 
\item
$\mathcal G^{m\times n}$, $\mathbb R^{m\times n}$, and $\mathbb C^{m\times n}$
 denote the classes of $m\times n$ Gaussian, real, or complex matrices, respectively.
\item
$\mathcal G_{m,n,r}$, $\mathbb R_{m,n,r}$, and $\mathbb C_{m,n,r}$,
for $1\le r\le \min \{m,n\}$,
 denote the classes of $m\times n$ matrices $M=UV$
(of rank at most $r$)
where both $m\times r$ matrix $U$  and $r\times n$ matrix $V$
are Gaussian, real, and complex, respectively.
\item
If   $U\in \mathcal G^{m\times r}$  and $V\in \mathcal G^{r\times n}$,
then we call $M=UV$ an
 $m\times n$ {\em factor-Gaussian
matrix of expected rank} $r$.
(In this case
the matrices $U$, $V$ and $M$  have rank $r$ 
with probability 1
 by virtue of Theorem \ref{thrnd}.)
 \end{itemize}
   

\subsection{The basic algorithm}\label{sbsalg}


A matrix $M$ can be represented (respectively, approximated) 
by a product $UV$ of two matrices $U\in \mathbb C^{m\times r}$ 
and $V\in \mathbb C^{r\times n}$ if and only if $r\ge \rank(M)$ 
(respectively, $r\ge \nrank(M)$),
and our main goal is the
 computation of such
a representation or approximation. 

We begin with the following basic algorithm
for the {\em fixed rank problem},
where the integer $r=\nrank(M)$ or
$r=\rank(M)$ is known. 
Otherwise
we can compute it  by means of
binary search based on recursive application of the algorithm
or proceed, e.g., as in our Algorithm \ref{alg11} of Section \ref{strrs}. 

\begin{algorithm}\label{alg1} {\rm Range Finder (See Figure 1
and compare \cite[Algorithms 4.1 and 4.2]{HMT11}).}


\begin{description}


\item[{\sc Input:}] 
An $m\times n$ matrix  $M$, a nonnegative tolerance $\tau$, and an integer  $r$
such that $0<r\ll \min\{m,n\}$.




\item[{\sc Initialization:}] 
 Fix an 
 integer $l$ such that $r\le l\ll \min\{m,n\}$. 
Generate an $n\times l$  matrix $B$.


\item[{\sc Computations:}]

\begin{enumerate}
\item 
Compute the  $m\times l$ matrix $MB$. Remove its columns that have small norms.
\item 
Orthogonalize its 
remaining columns (cf. \cite[Theorem 5.2.3]{GL13}), 
compute and output
the resulting $m\times \bar l$ 
matrix $U=U(MB)$ where $\bar l\le l$ .
\item 
Estimate the error norm 
$\Delta=||\tilde M-M||$ for $\tilde M=UU^TM$.

If $\Delta\le \tau$, output SUCCESS; otherwise
 FAILURE.
\end{enumerate}


\end{description}


\end{algorithm}


\begin{figure}[htb] 
\centering
\includegraphics[scale=0.25] {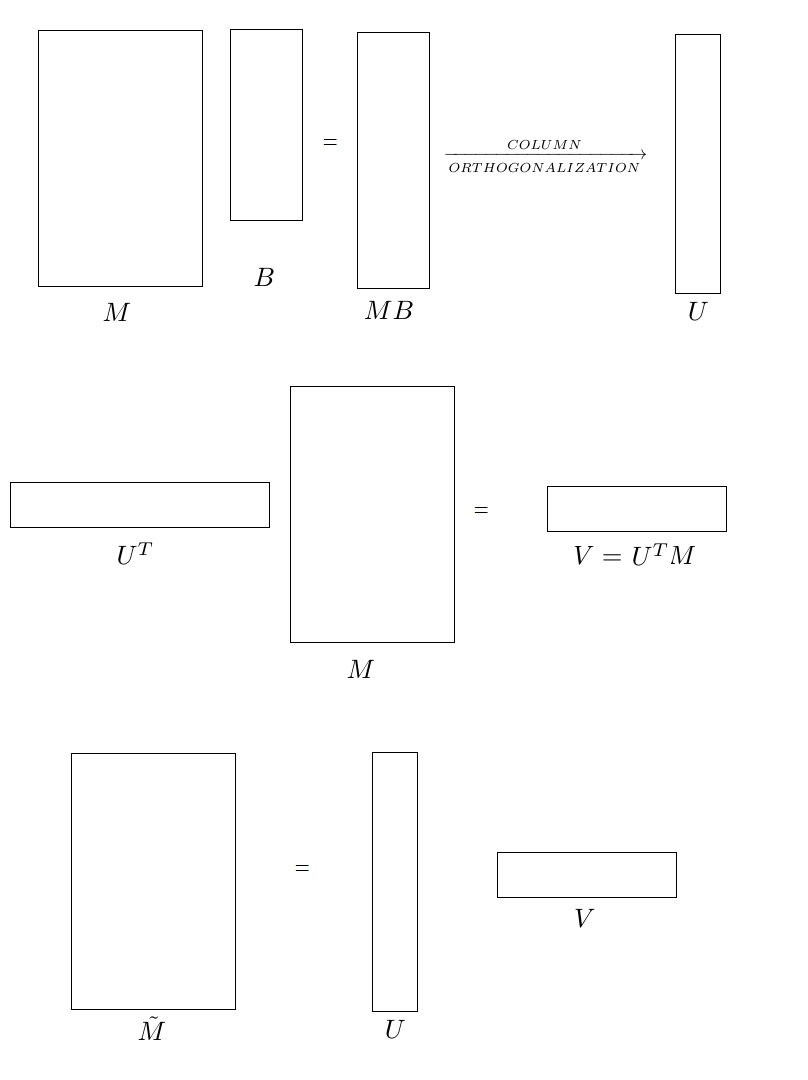}
\label{Fig01}
\caption{Matrices of Algorithm \ref{alg1}}
\end{figure}


At Stage 3 
probabilistic  
estimate for the norm $\Delta$ can be 
given by the norm $||\tilde MH-MH||$
for a random $n\times k$ matrix $H$ and a reasonably small 
positive integer $k$ (this would
extend
the Frievalds' 
probabilistic test of \cite{F77} 
(cf. \cite{MR95}, \cite{AS00})).
\cite[Algorithm 4.2]{HMT11}
 chooses  random matrix $B$  at Stage 1 
and then proceeds with 
 $k=l$ and $H=B$.

If  Stage 3 outputs SUCCESS, then   
a rank-$\bar l$ approximation to  the matrix $M$
is given by the matrix $\tilde M=UU^TM$,\footnote{By applying rank-revealing orthogonalization at Stage 2 
we can remove some extraneous columns and obtain
rank-$r$ approximation of the matrix $M$.}
 but the papers
 \cite[Section 5]{HMT11} and \cite{CW13} 
avoid costly multiplication of 
 $U^T$ by $M$. Their alternative solution relies on
the
extension of
 Stage 1 to randomized approximation
of the leading part 
of the compact SVD of the matrix $M$,
associated with its 
$r$  largest singular values.

The complexity of these algorithms is dominated at Stage 1,
which uses 
 $(2n-1)ml$ flops
in the case of
generic matrices $M$ and $B$.
With intricate application of {\em sparse embedding multipliers} $B$,
the paper  \cite{CW13} computes a low-rank 
approximation (with failure probability at most 1/5)
by using about $2mn$ flops for generic input $M$,
but alternative computations using $O(mnl)$ flops 
can be still of interest (see Section \ref{sextprg}).

 

\subsection{The choice of multipliers: basic observations}\label{smlbsc}


We readily verify the following  theorem (see Section \ref{sprth1}):

\begin{theorem}\label{thall}
Given an $m\times n$ matrix $M$ with $\nrank(M)=r$
 and a reasonably small positive tolerance  
$\tau$,  
Algorithm \ref{alg1} outputs SUCCESS
if and only if $\nrank(MB)
=r$. 
\end{theorem}

\begin{definition}\label{defbd}
For two integers $l$ and $n$, $0<l\le n$,
and any fixed 
 $n\times l$ multiplier $B$,
partition the set
 of  $m\times n$ matrices $M$
with $\nrank (M)=r$
into the set $\mathcal M_B=\mathcal M_{B,{\rm good}}$ 
of ``$B$-good" matrices such that $\nrank (MB)=r$
and the set
$\mathcal M_{B,{\rm bad}}$
of ``$B$-bad" 
matrices such that $\nrank (MB)<r$.
\end{definition}

The following simple observations should be instructive. 

\begin{theorem}\label{thbd} (Cf. Remark \ref{rewcd}.)
Consider a vector ${\bf v}$ of dimension $n$, 
an $n\times n$
 unitary matrix $U$,  and an $n\times l$ 
unitary  matrix
 $B$, so that $n\times l$ matrix
$UB$ is
 unitary. Then 

(i)  $\mathcal M_{UB}=(\mathcal M_B)U$, that is,
the map $B\rightarrow UB$ multiplies
the class $\mathcal M_B$ of $B$-good $m\times n$ matrices 
by the  unitary matrix $U$, 
 
(ii)  $\mathcal M_{B}\subseteq \mathcal M_{(B~|~{\bf v})}$,
that is, appending a column to 
a multiplier $B$ can only expand 
the class $\mathcal M_{B}$, 
 and 

(iii) this class fills the whole space $\mathbb C_{m,n,r}$
or  $\mathbb R_{m,n,r}$ if $l=n$.
\end{theorem}

\begin{proof}
Part (i) follows  because $(MU)B=M(UB)$.
Part (ii) follows  because  $\nrank(MB)\le \nrank(M(B~|~{\bf v}))$. 
Part (iii) follows  because  $\nrank (MB)=\nrank(M)$ if $B$ is 
an $n\times n$ unitary matrix.
\end{proof}


\subsection{\bf A recursive algorithm}\label{strrs}


Based on Theorem \ref{thbd} we devise the following algorithm
where $\nrank(M)$ is not known.

\begin{algorithm}\label{alg11} 
{\rm  Recursive low-rank representation/approximation of a matrix.} See Figure 2
and  cf. \cite[Algorithm 4.2]{HMT11}.


\begin{description}


\item[{\sc Input:}] 
An $m\times n$ matrix  $M$ and a nonnegative tolerance $\tau$.




\item[{\sc Computations:}]

\begin{enumerate}
\item 
Generate an $n\times n$  unitary matrix $\widehat B$.
\item 
Fix positive integers $l_1,\dots,l_h$
such that $l_1+\cdots+l_h=n$
(in particular $l_j=1$ for all $j$ if $h=n$)
and represent the matrix $\widehat B$ as 
a block vector $(B_1~|~B_2~|~\dots~|~B_h)$
where the block $B_i$ has size $n\times l_i$ for $i=1,\dots,h$.
\item 
Recursively, for $i=1,2,\dots$, apply Algorithm \ref{alg1}
to the matrix $M$ by
substituting $l^{(i)}=\sum_{j=1}^il_j$ for $l$ 
and $B^{(i)}=(B_1~|~B_2~|~\dots~|~B_i)$ for $B$. Stop when 
the algorithm outputs SUCCESS.
\end{enumerate}


\end{description}


\end{algorithm}


\begin{figure}[htb] 
\centering
\includegraphics[scale=0.5] {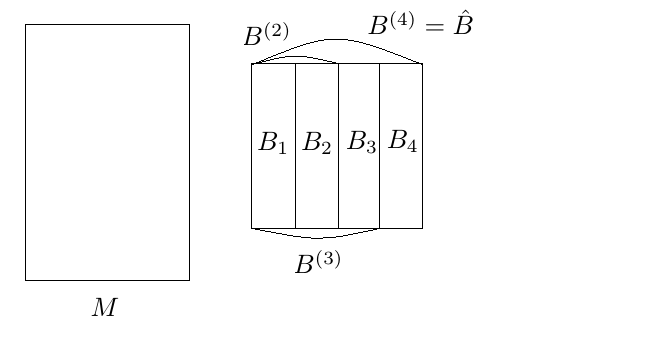}
\label{Fig05}
\caption{Matrices of Algorithm \ref{alg11}}
\end{figure}


By virtue of part (iii)
of Theorem \ref{thbd}
the algorithm stops and outputs SUCCESS either at the $h$th Stage 
(when $l^{(h)}=n$) or earlier, and we are surely interested in  yielding
SUCCESS already for $l^{(i)}\ll n$ and in saving flops for 
matrix multiplications at Stage 3.


\begin{remark}\label{resprs01}
We are likely  to save some flops if we compute  
 approximate matrix products
by using
leverage scores \cite{W14}
(a.k.a. sampling probabilities \cite[Sections 3 and 5]{M11}).
\end{remark}


\begin{remark}\label{rewcd}
Clearly, the blocks $B_j$ and $B^{(i)}$ of the unitary  matrix $\widehat B$
are unitary as well, but we can readily 
extend both Theorem \ref{thbd} and Algorithm \ref{alg11} 
to the case where we 
apply them to a nonsingular and well-conditioned
 (rather than unitary)
 $n\times n$
matrix $\widehat B$.
In that case all multipliers $B^{(i)}$ and  all their  blocks $B_j$ 
are also well-conditioned
matrices of full rank, 
and moreover  $\kappa(B^{(i)})\le \kappa(B)$ and
$\kappa(B_j)\le \kappa(B)$ for all $i$ and $j$
(cf. \cite[Corollary 8.6.3]{GL13}).
\end{remark}  


\subsection{\bf Benefits of
using  
Gaussian and random structured 
multipliers}\label{scmpgrs}


\begin{theorem}\label{th0}  
Let  Algorithm \ref{alg1} be applied with a Gaussian
multiplier $B\in \mathcal G^{n\times l}$. Then 

(i)  $\tilde M=M$ with probability 1
if $l\ge r=\rank (M)$ 
(cf. Theorem \ref{thlrnk})
and 

(ii) it is likely that $\tilde M\approx M$
if $\nrank (M)=r\le l$,
and the probability that $\tilde M\approx M$ 
approaches 1 fast as $l$ increases from $r+1$
(cf. Theorem \ref{th1}).
\end{theorem}


The  theorem implies that
 Algorithm  \ref{alg11} is likely
to output SUCCESS at Stage $h$ for the smallest $h$ such that
$l^{(h)}\ge r$ in the case where  $B$ denotes a Gaussian 
(rather than unitary) matrix.

 
An $n\times l$ matrix of {\em subsample random Fourier or 
Hadamard transform}\footnote{Hereafter we use the acronyms {\em SRFT}  and {\em SRHT}.}
is  defined
 by $n+l$ random variables (see Remark \ref{resrft}), and we can   
pre-multiply it by a vector by using $O(n\log(l))$ flops, 
for $l$ of order $r\log (r)$,
(see \cite[Sections 4.6 and 11]{HMT11},
 \cite[Section 3.1]{M11}, and  \cite{T11}).
 For comparison,
  an $n\times l$ Gaussian matrix is  defined by its $nl$ random entries,
 and we need  $l(2n-1)$ flops in order 
to pre-multiply it by a vector.

SRFT and SRHT multipliers $B$ are {\em universal}, like Gaussian ones: 
 Algorithm \ref{alg1} applied with such a multiplier is likely to 
approximate 
closely 
a matrix $M$ having numerical rank at most $r$, 
although the estimated failure probability 
$3\exp(-p)$, for $p=l-r\ge 4$ with  Gaussian multipliers
 increases to order of $1/l$ in the case of
SRFT and SRHT  multipliers (cf. 
\cite[Theorems 10.9 and 11.1]{HMT11}, \cite[Section 5.3.2]{M11}, and
 \cite{T11}). 

Empirically  Algorithm   \ref{alg1} 
with SRFT  multipliers 
fails very rarely even for 
$l=r+20$, 
although for some special input matrices $M$
it is likely to fail 
if $l=o(r\log (r))$ 
(cf. 
 \cite[Remark 11.2]
{HMT11} or  \cite[Section 5.3.2]
{M11}). 
Researchers have consistently observed
similar empirical behavior 
of the algorithm  
 applied with SRHT and various other
 multipliers
(see \cite{HMT11}, \cite{M11}, \cite{W14}, \cite{PQY15},
and the references therein),\footnote{In view of part (i) of Theorem \ref{thbd},
the results cited  for the classes of SRFT and SRHT
matrices also hold for the products of these classes with any unitary matrix,
in particular for the class of 
 $n\times l$ submatrices
of an $n\times n$ circulant matrix, each made up of
  $l$ randomly chosen  columns
(see  
Remark \ref{resrft}).}
 but
so far no adequate formal support for that  empirical observation
has 
appeared in the huge bibliography 
on this highly popular subject.


\subsection{\bf Our goals, our dual theorem, and its implications}\label{sglsdl}


In this paper we are going to

(i) fill the void in the bibliography by supplying a missing
{\em formal support} for the cited observation, 
with far reaching implications (see parts (ii) and (iv) below), 

 (ii) define  {\em new more efficient policies of generation and application of
multipliers}  for low-rank approximation,

(iii) {\em test our policies numerically}, and 

(iv) {\em extend our progress} 
to other important areas of matrix  computations.

\medskip

Our {\em Dual} Theorem \ref{th0d} below  
reverses the assumptions of our {\em Primal} Theorem \ref{th0}
that a multiplier $B$ is Gaussian, 
while a matrix $M$ 
is fixed. 

\begin{theorem}\label{th0d} 
Let $M-E\in \mathcal G_{m,n,r}$ and 
 $||E||_2\approx 0$ (in which case 
$\nrank (M)\le r$ and,  with a probability close to 1,
 $\nrank (M)=r$). Furthermore let
$B\in \mathbb R^{n\times l}$  and  
$\nrank (B)=l$. Then \\
(i)  Algorithm \ref{alg1}
outputs a rank-$r$ representation of a matrix $M$ 
with probability 1
if $E=0$ and if $l=r$, and \\
(ii) it 
outputs a rank-$l$ approximation of that matrix 
with a probability close to 1 if $l\ge r$
and approaching 1 fast as the integer $l$ increases from $r+1$.
\end{theorem}
\begin{proof}
See Theorem \ref{th1d}.
\end{proof}

\begin{corollary}\label{co0d}
     Under the assumptions of Theorem \ref{th0d},
Algorithm \ref{alg11}
is likely to produce a rank-$l$ approximation 
to the matrix $M$ at its first Stage $i$ 
at which  $l^{(i)}\ge r$,
and the probability that this occurs
approaches 1 fast as $l^{(i)}$ increases from $r+1$. 

\end{corollary}




Part (ii) of Theorem \ref{th0d} implies that 
{\em Algorithm \ref{alg1} succeeds for the average input 
matrix $M$ that has a small numerical rank} $r\le l$
(and thus in a sense to most of such matrices)
if the multiplier $B$ is any unitary  matrix (or even
 any well-conditioned matrix
of full rank) and if  the
 average matrix is defined under the 
Gaussian probability distribution.
The former provision, that $\nrank(B)=l$, is  natural
for otherwise we could have replaced the multiplier $B$
by an $n\times l_-$ matrix for some integer $l_-<l$. 
 The latter  customary provision is natural
in view of the Central Limit Theorem.

For an immediate implication of Theorem \ref{th0d}, 
on the average input $M$ having numerical rank at  most $r$,
Algorithm \ref{alg11} applied with any unitary 
or even  any nonsingular and
well-conditioned $n\times n$ multiplier $B$
outputs SUCCESS 
at its  earliest recursive Stage $i$
at which the dimension $l^{(i)}=\sum_{j=1}^i l_j$ exceeds $r-1$.
This
 can be viewed as {\em derandomization} of 
Algorithms \ref{alg1} and \ref{alg11} 
versus their  application with  Gaussian sampling.


\subsection{\bf Related work, our novelties, and extension of our progress}\label{sextprg}


Part (ii) of our Theorem   \ref{th0} is
implied by \cite[Theorem 10.8]{HMT11}, but our specific supporting estimates are 
 more compact, cover the case of any $l\ge r$
(whereas \cite{HMT11} assumes that $l\ge r+4$),
and we deduce them by using a shorter proof (see Remark \ref{repfr}).    
Our approach and our results of Section \ref{sglsdl} 
are new, and so are our
families of sparse and structured multipliers and the
policies of their generation, combination, and application 
in Sections \ref{smngflr} and \ref{ssprsml} as well.
  
By applying the well-known links, 
we can extend our results for low-rank approximation
to various fundamental problems 
of matrix  computations and data mining and analysis,
but  our duality techniques can be 
extended to other important computational problems as well.
In Section \ref{sext} (Conclusions) we
show such an extension 
to the Least Squares Regression.\footnote{Hereafter we use the acronym 
``LSR".} 
Another extension in \cite{PZa}
supports numerically safe performance of
Gaussian elimination with no pivoting\footnote{Pivoting, 
that is, row or column interchange of an input matrix
for avoiding numerical problems
in Gaussian elimination,   
 is communication intensive and has become 
 the bottleneck of Gaussian elimination 
 in the present day computer environment.
Preprocessing with randomized and derandomized 
multipliers is a natural means for overcoming this problem
(cf.  \cite{BBBDD14},  \cite{PQY15}, \cite{PZa}).}
 and 
block Gaussian elimination.
The extensions
 provide new insights
and new opportunities and should motivate 
further effort and further progress.
 
Extensive decade-long work of a number 
of authors on an alternative approach 
to Low-Rank Approximation
and Least Squares Regression
has culminated in the paper \cite{CW13}.
Its algorithms succeed for 
these problems  with a probability
 at least 4/5,
whereas we only reach
solution for the average input. 
Our study, however,   
 leads to some  benefits, which should 
compensate for  this deficiency.

1. We show the power of a very large class of multipliers,
including various sparse and structured ones.
This can be interesting, e.g., for some special structured inputs
(see Remark \ref{restr}),  
but not only for them. 
Indeed, see our Remark \ref{resprs0}
and compare the following excerpt from \cite{BCDHKS14}:
``The traditional metric for the efficiency of a numerical algorithm has been the number of arithmetic operations it performs. Technological trends have long been reducing the time to perform an arithmetic operation, so it is no longer the bottleneck in many algorithms; rather, communication, or moving data, is the bottleneck".

2. In order to make the probability of failure less than $\delta$, 
the complexity bound of  \cite{CW13} involve overhead 
of order $\log(1/\delta))$, 
which greatly exceeds the overhead in the case of 
our average case estimates.

3.  \cite{CW13} studies the fixed rank problem;
in the case where the input numerical rank 
is not known, our Algorithm \ref{alg11} 
substantially
decreases the computational   overhead 
versus binary search.

4. Unlike \cite{CW13}  we cover the case where  
the ratio $n/r$ is not very large, which 
can still be interesting in some applications.
 
5. Our
  analysis  is quite simple and 
conceptually distinct and should be of independent interest
because it provides 
elusive explanation of 
a well-known empirical phenomenon (cf. Section \ref{sprbpr}). 


\subsection{\bf Organization of the paper}\label{sorg}


We organize our presentation as follows: 

\begin{itemize}
\item
In Section \ref{smngflr}
 we describe our policies for
management of the rare failures of Algorithm \ref{alg1}
and  amend Algorithm \ref{alg11}.
\item
In Section \ref{ssprsml} we present some efficient sparse and structured multipliers 
for low-rank approximation.
\item
In Section \ref{sprf} we prove
 Theorems \ref{thall}, \ref{th0}, and \ref{th0d},
extending their claims with more detailed estimates. 
\item
Section \ref{ststs}
(the contribution of the second and the third authors)
covers our numerical tests.
\item
In Section \ref{sext} we extend our approach
to the LSR computations.
\item
The Appendix covers some auxiliary results
for computations with random matrices. 

\end{itemize}


\section{Preventing and managing the unlikely failure of Algorithm \ref{alg1}}\label{smngflr} 




{\bf Our conflicting goals and simple recipes.}

We try to decrease:

(i) the cost of the generation of a multiplier $B$ and of the computation of the 
matrix $MB$, 

(ii)  the chances for the failure of Algorithm \ref{alg1}, and

(iii)  the rank of the computed approximation of a matrix $M$.

Towards  goal (i) we propose using
sparse and structured $n\times l$ multipliers 
in the next section.
They are pre-multiplied by a vector and by a matrix $M$ at
a low cost even for $l=n$. 

Towards  goal (ii)  we can expect to succeed 
whenever integer parameter $l$ exceeds $r+1$,
but our chances for success grow fast as $l$ increases.
Such an increase is in conflict with our goal (iii),
but we can alleviate the problem by using the following simple technique.

\medskip 

{\bf Randomized Compression Algorithm} (see Figure 3).
\begin{enumerate}
\item
Fix a sufficiently large dimension $l$, which is still much smaller than $\min\{m,n\}$,
generate a sparse and structured   
$n\times l$ multiplier $B$, and compute the $m\times l$ product $MB$ (by
using $(2n-1)ml$ flops).
\item
Fix a smaller integer $l_-$ such that $r\le l_-< l$, 
generate a Gaussian $l\times l_-$
multiplier $G$, and compute and output
the $m\times l_-$ matrix $MBG$ (by using $(2l-1)ml_-$ flops, dominated at Stage 1
if $l\le\min\{m,n\}$). 
\end{enumerate}

\begin{figure}[htb] 
\centering
\includegraphics[scale=0.5] {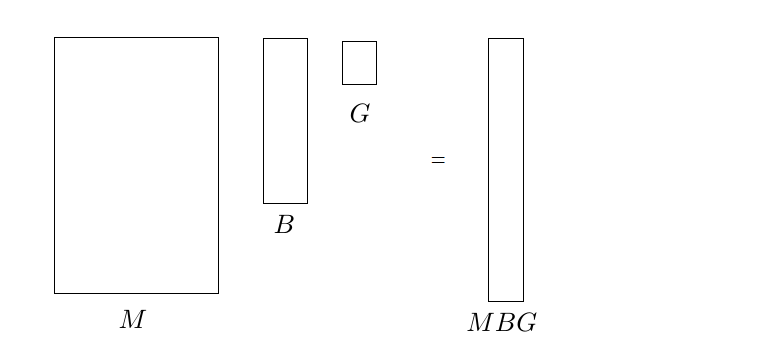}
\label{Fig04}
\caption{Matrices of Algorithm \ref{alg11}}
\end{figure}

By virtue of Theorem \ref{th0} this algorithm is likely
to succeed, but
in our extensive tests in  Section \ref{ststs} even
the following simple  heuristic recipe has always worked.

\medskip 

{\bf Heuristic Compression Algorithm} (linear combination of failed multipliers): 
if the first  $h$ recursive steps of Algorithm \ref{alg11} have failed 
for $h>1$,
then apply  Algorithm \ref{alg1} with a multiplier
$B=\sum_{j=1}^hc_jB_j$ where $c_j=\pm 1$ for all $j$
and for a fixed or random choice of the signs $\pm$.
(More generally, one can choose complex values $c_j$
on the unit circle, letting $|c_j|=1$ for all $j$.)

\begin{remark}\label{respectsts}
In our study above we rely on the results of Theorem \ref{th0d},
which cover the average input matrices. Real computations
 can deal with  ``rare special" 
input matrices $M$, not covered by Theorem \ref{th0d},
but in our tests in Section \ref{ststs}
with a variety
of inputs, our small collection of sparse and structured multipliers 
of the next section
turned out to be powerful enough for handling various important classes of special 
matrices as well. 
\end{remark}

\begin{remark}\label{rereus} {\em Reusing multipliers.}
Recall from 
\cite[Sections 8 and 9]{HMT11} that for all $j$
the matrices
$\tilde M^{(j)}=U^{(j)}U^{(j)T}M$
 and $\tilde M_j=U_jU_j^TM$, for $U^{(j)}=U(MB^{(j)})$,
 $U_j=U(MB_j)$, and $B^{(j)}=(B_1~|~B_2~|~\dots~|~B_j)$,
are the orthogonal projections 
of the matrices 
$MB^{(j)}$ and $MB_j$, 
respectively, onto the range of the matrix $M$.
Hence  
 $M-\tilde M^{(h)}=M-\tilde M^{(h-1)}-\tilde M_h$,
and so  at the $h$-th stage of Algorithm \ref{alg11}, for $h>1$,
we can reuse such projections computed at its Stage $h-1$
rather than recompute them.
\end{remark}


\section{Generation of multipliers. Counting flops and random variables}\label{ssprsml}


In our tests we have consistently succeeded by using multipliers
from a {\em limited family} of very sparse and highly structured
orthogonal matrices of classes 13--17 of Section \ref{s17m}, but 
in this section
we  also cover a greater variety of  other sparse and structured matrices,
 which form an {\em extended family}
of multipliers.

We proceed in the following order.
Given 
two integers $l$ and $n$,  $l\ll n$, we first generate
four classes of  very sparse primitive $n\times n$ unitary matrices,
then combine them into some basic families of $n\times n$ matrices 
(we denote them $\widehat B$ in this section),
and finally define multipliers  $B$ as  $n\times l$  submatrices 
made up of $l$ columns, which can be fixed (e.g., leftmost) or 
chosen at random. 
The matrix $B$ is unitary if so is the matrix 
$\widehat B$, and more generally
 $\kappa(B)\le \kappa(\widehat B)$ 
(cf. \cite[Theorem 8.6.3]{GL13}).     


\subsection{$n\times n$ matrices of four  primitive types}\label{sdfcnd}

  



\begin{enumerate}
  \item
A fixed  or random
 {\em permutation matrix} $P$.
Their block submatrices form the important class of
 CountSketch matrices from
the data stream literature (cf. \cite[Section 2.1]{W14},
\cite{CCF04}, \cite{TZ12}).
\item
A {\em diagonal matrix}  $D=\diag(d_i)_{i=0}^{n-1}$, with 
fixed  or random
diagonal entries $d_i$ such that
$|d_i|=1$ for all $i$ (and so all $n$ entries $d_i$ lie 
on the unit circle $\{x:~|z|=1\}$, 
being either nonreal or  $\pm 1$).
\item
An $f$-{\em circular shift matrix} 

$$Z_f=\begin{pmatrix}
        0  ~ & \dots&     ~  ~\dots   ~ ~& 0 & ~f\\
        1   & \ddots    &   &    ~ 0 & ~0\\
        \vdots         & \ddots    &   \ddots & ~\vdots &~ \vdots  \\
         \vdots   &    &   \ddots    &~ 0 & ~0 \\
        0   & ~  \dots &      ~ ~~  \dots   ~~ ~ ~ & 1 & ~0 
    \end{pmatrix}
$$
and its transpose $Z_f^T$ for a scalar $f$ such that either $f=0$ or $|f|=1$.
We write $Z=Z_0$, call $Z$ {\em unit down-shift matrix}, and call the special permutation 
matrix $Z_1$ the
 {\em unit circulant matrix}. 
\item
A   $2s\times 2s$ {\em Hadamard  primitive matrix}
$H^{(2s)}=\big (\begin{smallmatrix} I_s  & ~I_s  \\
I_s   & -I_s\end{smallmatrix}\big )$
 for a positive integer $s$
 (cf. \cite{M11},  \cite{W14}).
\end{enumerate}
The latter primitive  $n\times n$ matrices are very sparse,  have
nonzero entries evenly distributed 
throughout,  and can be
pre-multiplied by a vector
by using from 0 to $2n$ flops.

 All our primitive matrices,
 except for the matrix $Z$, are unitary or real orthogonal. Hence,
for  the average input matrix $M$, 
Algorithm \ref{alg1} succeeds with any of their $n\times l$ submatrix $B$
by virtue of  Theorem \ref{th0d},
and similarly with any $n\times l$ submatrix of the matrix $Z$ of full rank $l$.

For specific input matrices the algorithm can fail
with some of our $n\times l$ primitive multipliers $B$
 (e.g., this is  frequently the case 
where both input matrix $M$ and multiplier $B$ 
are  sparse), 
 but in the next subsections 
we 
readily combine primitives 1--4 into
families of $n\times n$ 
sparse and/or structured  matrices,
and in Section \ref{ststs} we
consistently and successfully test
their $n\times l$ submatrices $B$
as multipliers.
   

\subsection{Family (i): 
multipliers based on the Hadamard and Fourier processes}\label{shad}


  At first we recall the  following recursive definition of  
dense and orthogonal (up to scaling by constants) $n\times n$ matrices $H_n$ 
of {\em Walsh-Hadamard transform} for $n=2^k$
(cf.   \cite[Section 3.1]{M11} and our Remark \ref{recmb}): 
\begin{equation}\label{eqrfd}
H_{2q}=\begin{pmatrix}
H_{q} & H_{q} \\
H_{q} & -H_{q}
  \end{pmatrix}
\end{equation}
for $q=2^h$, $h=0,1,\dots,k-1$,
  and the Hadamard  primitive matrix
 $H_{2}=H^{(2)}=
\big (\begin{smallmatrix} 1  & ~~1  \\
1   & -1\end{smallmatrix}\big )$  of type 4
 for $s=1$.  
 
For demonstration, here are the matrices $H_4$ and $H_8$
shown with their entries,  
$$H_{4}=\begin{pmatrix}
1  &  1 & 1  &  1 \\
1  & -1  & 1  &  -1 \\ 
1  &  1 & -1  &  -1 \\
1  & -1  & -1  &  1
\end{pmatrix}~{\rm and}~ H_{8}=\begin{pmatrix}
1  &  1 & 1  &  1 & 1  &  1 & 1  &  1 \\
1  & -1  & 1  &  -1 & 1  & -1  & 1  &  -1\\ 
1  &  1 & -1  &  -1 & 1  & 1  & -1  &  -1\\
1  & -1  & -1  &  1 &  1  & -1  & -1  &  1 \\
1  &  1 & 1  &  1 & -1  &  -1 & -1  &  -1 \\
1  & -1  & 1  &  -1 & -1  &  1 & -1  &  1\\ 
1  &  1 & -1  &  -1 & -1  &  -1 & 1  &  1\\
1  & -1  & -1  &  1 & -1  &  1 & 1  &  -1
\end{pmatrix},$$
 but for larger dimensions $n$, 
recursive representation (\ref{eqrfd}) enables much faster 
 pre-multiplication of a matrix $H_{n}$ by a vector, namely it is sufficient
to use $nk$ additions
and subtractions for $n=2^k$.

Next we sparsify this matrix by defining it by a 
 shorter recursive process, that is, 
by fixing a  {\em recursion depth} $d$, $1\le d<k$, and  applying equation (\ref{eqrfd}) where
$q=2^hs$, $h=k-d,k-d+1,\dots,k-1$,  and $H_{s}I_s$ for 
$n=2^ds$.
For  two positive integers $d$ and $s$,
 we denote the resulting $n\times n$ matrix
$H_{n,d}$ and  for $1\le d< k$ call it 
 $d$--{\em Abridged Hadamard
 (AH)
  matrix}. 
In particular, 
$$H_{n,1}=\begin{pmatrix}
I_s  &  I_s  \\
I_s  & -I_s  
\end{pmatrix},
{\rm for}~n=2s;~
H_{n,2}=\begin{pmatrix}
I_s  &  I_s & I_s  &  I_s \\
I_s  & -I_s  & I_s  &  -I_s \\ 
I_s  &  I_s & -I_s  &  -I_s \\
I_s  & -I_s  & -I_s  &  I_s
\end{pmatrix},~{\rm for}~n=4s,~{\rm and}$$
$$H_{n,3}=\begin{pmatrix}
I_s  &  I_s & I_s  &  I_s & I_s  &  I_s & I_s  &  I_s \\
I_s  & -I_s  & I_s  &  -I_s & I_s  & -I_s  & I_s  &  -I_s\\ 
I_s  &  I_s & -I_s  &  -I_s & I_s  & I_s  & -I_s  &  -I_s\\
I_s  & -I_s  & -I_s  &  I_s &  I_s  & -I_s  & -I_s  &  I_s \\
I_s  &  I_s & I_s  &  I_s & -I_s  &  -I_s & -I_s  &  -I_s \\
I_s  & -I_s  & I_s  &  -I_s & -I_s  &  I_s & -I_s  &  I_s\\ 
I_s  &  I_s & -I_s  &  -I_s & -I_s  &  -I_s & I_s  &  I_s\\
I_s  & -I_s  & -I_s  &  I_s & -I_s  &  I_s & I_s  &  -I_s
\end{pmatrix},~{\rm for}~n=8s.$$
For a fixed $d$, the matrix  $H_{n,d}$
is still orthogonal up to scaling,
 has $q=2^d$ nonzero entries 
in every row and  column, and hence
 is sparse unless 
$k-d$ is a small integer.

Then again,
for larger dimensions $n$, 
 we can pre-multiply such a  matrix by a vector much faster
if, instead of the representation by its entries, we apply
recursive process (\ref{eqrfd}), which involves just
  $dn$
additions
and subtractions
and allows
highly efficient  
 parallel implementation 
(cf. 
Remark \ref{resprs0}).

We similarly obtain sparse matrices
by shortening a recursive process
of the generation of the $n\times n$ matrix  $\Omega_n$
of {\em discrete Fourier transform (DFT)} at $n$ points,  for $n=2^k$:
\begin{equation}\label{eqdft}
\Omega_n=(\omega_{n}^{ij})_{i,j=0}^{n-1},~{\rm for}~n=2^k~{\rm and~a~primitive}~
n{\rm th~root~of~unity}~\omega_{n}=\exp(2\pi {\bf i}/n),~{\bf i}=\sqrt {-1}.
\end{equation}
In particular
$\Omega_{2}=H^{(2)}=
\big (\begin{smallmatrix} 1  & ~~1  \\
1   & -1\end{smallmatrix}\big )$,
$$\Omega_{4}=\begin{pmatrix}
1  &  1 & 1  &  1 \\
1  & {\bf i}  & -1  &  -{\bf i} \\ 
1  &  -1 & 1  &  -1 \\
1  & -{\bf i}  & -1  &  {\bf i}
\end{pmatrix},~{\rm and}~ \Omega_{8}=\begin{pmatrix}
1  &  1 & 1  &  1 & 1  &  1 & 1  &  1 \\
1 &\omega_{8}&{\bf i}&{\bf i}\omega_{8}&-1& -\omega_{8}& -{\bf i}&-{\bf i}\omega_{8}\\ 
1  &  {\bf i} & -1  &  -{\bf i} & 1  & {\bf i}  & -1  &  -{\bf i}\\
1 &{\bf i}\omega_{8}& -{\bf i}&\omega_{8}& -1 &-{\bf i}\omega_{8}&{\bf i}&-\omega_{8} \\
1  &  -1 & 1  &  -1 & 1  &  -1 & 1  &  -1 \\
1  & -\omega_{8}& {\bf i}&-{\bf i}\omega_{8}& -1&\omega_{8}& -{\bf i}& {\bf i}\omega_{8}\\ 
1  &  - {\bf i} & -1  &   {\bf i} & 1  &  - {\bf i} & -1  &   {\bf i}\\
1  & - {\bf i}\omega_{8}& -{\bf i}&-\omega_{8} & -1&{\bf i}\omega_{8}&{\bf i}& \omega_{8}
\end{pmatrix}.$$

The matrix $\Omega_n$ is unitary up to scaling by $\frac{1}{\sqrt n}$.
We can pre-multiply it by a vector 
by using $1.5nk$ flops, and we can efficiently parallelize this computation
 if, instead of representation by entries, we apply  following recursive representation
(cf. \cite[Section 2.3]{P01}  and our Remark \ref{recmb}):\footnote{This is a representation of FFT, called decimation in frequency (DIF) radix-2 representation. 
Transposition turns it into an alternative
representation of FFT, called
decimation 
in time (DIT) radix-2 representation.}
\begin{equation}\label{eqfd}
\Omega_{2q}=
\widehat P_{2q}
\begin{pmatrix}\Omega_{q}&~~\Omega_{q}\\ 
\Omega_{q}\widehat D_{q}&-\Omega_{q}\widehat D_{q}\end{pmatrix},~
\widehat D_{q}=\diag(\omega_{n}^{i})_{i=0}^{n-1}.
\end{equation}
Here $\widehat P_{2q}$ is the matrix of odd/even permutations 
 such that 
$\widehat P_{2^{h}}({\bf u})={\bf v}$, ${\bf u}=(u_i)_{i=0}^{2^{h}-1}$, 
${\bf v}=(v_i)_{i=0}^{2^{h}-1}$, $v_j=u_{2j}$, $v_{j+2^{h-1}}=u_{2j+1}$, 
$j=0,1,\ldots,2^{h-1}-1$;
$q=2^h$, $h=0,1,\dots,k$,  and
$\Omega_{1}=(1)$ is the scalar 1.

We sparsify this matrix by defining it by a 
 shorter recursive process, that is, 
by fixing a recursion depth $d$, $1\le d<k$,  
replacing $\Omega_{s}$ for $s=n/2^{d}$
by the identity matrix $I_{s}$,
and then  applying equation (\ref{eqfd}) for
$q=2^h$, $h=k-d,k-d+1,\dots,k-1$.

For $1\le d<k$ and $n=2^ds$,   
we denote the resulting $n\times n$ matrix
$\Omega_{n,d}$ and call it 
 $d$-{\em Abridged   Fourier
 (AF)
  matrix}. It is also unitary (up to scaling),
 has $q=2^d$ nonzero entries 
in every row and column, and thus is
  sparse unless 
$k-d$ is a small integer. 
We can represent such a matrix by its entries, 
but then again its
pre-multiplication by a vector involves just  $1.5dn$
flops
and allows
highly efficient  
 parallel implementation
if we rely on recursive 
representation (\ref{eqfd}). 


\medskip

By applying fixed or random permutation and scaling to AH matrices
$H_{n,d}$ and  AF matrices $\Omega_{n,d}$, we obtain the 
families of 
$d$--{\em Abridged Scaled and Permuted 
 Hadamard (ASPH)} matrices, $PDH_n$, and 
$d$--{\em Abridged Scaled and Permuted 
Fourier  (ASPF)} $n\times n$
 matrices, $PD\Omega_n$ where $P$ and $D$ are 
two 
matrices of permutation and diagonal scaling of
primitive classes 1 and 2, respectively. 
Likewise  we define the families of
ASH, ASF, APH, and APF matrices, 
 $DH_{n,d}$, $D\Omega_{n,d}$, $PH_{n,d}$, and  $P\Omega_{n,d}$, respectively.
Each random permutation or scaling  
 contributes up to $n$ random parameters.  

\begin{remark}\label{recmb}
The following equations are equivalent  to (\ref{eqrfd}) and (\ref{eqfd}):
$$H_{2q}=\diag(H_q,H_q)H^{(2q)}~{\rm and}~\Omega_{2q}=
\widehat P_{2q} \diag(\Omega_{q},\Omega_{q}\widehat D_q)H^{(2q)}$$
where $H^{(2q)}$ denotes a $2q\times 2q$ Hadamard's primitive matrix of type 4.
By extending the latter recursive representation we can 
define matrices that involve more random parameters. Namely we can
 recursively  incorporate random  permutations and diagonal scaling as follows:
\begin{equation}\label{eqhfspd}
\widehat H_{2q}=P_{2q}D_{2q}\diag(\widehat H_q,\widehat H_q)H^{(2q)}~{\rm and}~
\widehat \Omega_{2q}=
P_{2q}D_{2q} \diag(\Omega_{q},\Omega_{q}\widehat D_q)H^{(2q)}.
\end{equation}
Here $P_{2q}$ are $2q\times 2q$ random permutation matrices of primitive class 1
and $D_{2q}$ are $2q\times 2q$ random matrices of diagonal scaling of primitive class 2,
for all $q$. Then again we define $d$--abridged matrices $\widehat H_{n,d}$
and $\widehat \Omega_{n,d}$ by applying only $d$ recursive steps (\ref{eqhfspd})
initiated at the primitive matrix $I_{s}$, for $s=n/2^{d}$.

With these recursive steps we can pre-multiply matrices $\widehat H_{n,d}$
and $\widehat \Omega_{n,d}$ by a vector
by using  at most $2dn$ additions and subtractions
and at most $2.5dn$ flops,
respectively, provided that $2^{d}$ divides
 $n$.
\end{remark}


\subsection{ $f$-circulant, sparse   
$f$-circulant,  and uniformly  sparse matrices}\label{scrcsp}


An
 {\em $f$-circulant matrix}
$$Z_f({\bf v})
=\begin{pmatrix}v_0&fv_{n-1}&\cdots&fv_1\\ v_1&v_0&\ddots&\vdots\\ \vdots&\ddots&\ddots&fv_{n-1}\\ v_{n-1}&\cdots&v_1&v_0\end{pmatrix}=\sum_{i=0}^{n-1}v_iZ_f^i$$
for  
the matrix $Z_f$ of $f$-circular shift,
is defined by a 
scalar $f\neq 0$  and by
the first column ${\bf v}=(v_i)_{i=0}^{n-1}$ and
 is called {\em  circulant} if $f=1$ and {\em skew-circulant} if $f=-1$.
Such a matrix is nonsingular with probability 1 (see Theorem \ref{thrnd}) and
is likely to be well-conditioned \cite{PSZ15}
if  $|f|=1$  and if the vector ${\bf v}$ is Gaussian.


\begin{remark}\label{restr}
One can compute the product of an $n\times n$ circulant  matrix with 
an $n\times n$ Toeplitz or Toeplitz-like matrix
by using $O(n\log (n))$ flops (see \cite[Theorem 2.6.4 and Example 4.4.1]{P01}).
\end{remark}

{\bf FAMILY (ii)}  of  {\em sparse} $f$-{\em circulant matrices} 
$\widehat B=Z_f({\bf v})$ is
 defined by a fixed or random scalar $f$, $|f|=1$, and by
the  first column having exactly 
$q$ nonzero entries, for $q\ll n$.
The positions and the values of nonzeros can be
 randomized (and then the matrix would depend on up to $2n+1$ random values).

Such a matrix can be pre-multiplied by a vector by using at most 
$(2q-1)n$ flops  or, in the real case where $f=\pm 1$ and $v_i=\pm 1$
for all $i$, by using at most
$qn$ additions and subtractions. 

The same cost estimates apply in the case of the  generalization 
of $Z_f({\bf v})$  to
a {\em uniformly   
 sparse matrix} with exactly $q$ nonzeros entries, $\pm 1$, 
in every row and in every column for $1\le q\ll n$.
Such a matrix
is the sum $\widehat B=\sum_{i=1}^q\widehat D_iP_i$
for fixed or random matrices  $P_i$  and $\widehat D_i$ of  primitive  types 1 and 2,
respectively.


\subsection{Abridged   
$f$-circulant matrices}\label{scrcabr}


First recall the following well-known expression for a
 $g$-circulant matrix: 
$$Z_g({\bf v})=\sum_{i=0}^{n-1}v_iZ_g^i=D_f^{-1}\Omega_n^HD\Omega_nD_f$$
where $g=f^n$,  $D_f=\diag(f^i)_{i=0}^{n-1}$, ${\bf v}=(v_i)_{i=0}^{n-1}=(\Omega_nD_f)^{-1}{\bf u}$, 
${\bf u}=(u_i)_{i=0}^{n-1}$, and
$D=\diag(u_i)_{i=0}^{n-1}$
(cf. \cite[Theorem 2.6.4]{P01}).
For $f=1$, the expression is simplified: $g=1$, $D_f=I_n$, and
$Z_g({\bf v})=\sum_{i=0}^{n-1}v_iZ_1^i$
is a circulant matrix:
\begin{equation}\label{eqcrcl}
Z_1({\bf v})=\Omega_n^HD\Omega_n,~D=\diag(u_i)_{i=0}^{n-1},~{\rm for}~ 
{\bf u}=(u_i)_{i=0}^{n-1}=\Omega_n{\bf v}.
\end{equation}
Pre-multiplication of an $f$-circulant matrix by a vector 
is reduced to pre-multiplication of each of the matrices $\Omega$
and $\Omega^H$ by a vector and in  addition 
to performing $4n$ flops (or $2n$ flops in case of a circulant matrix).
 This involves $O(n\log (n))$ flops overall
and then again allows highly efficient
 parallel implementation.

For a fixed scalar $f$ and $g=f^n$, we can define the matrix $Z_g({\bf v})$ by 
 any of the two vectors ${\bf u}$ or  ${\bf v}$. 
The matrix is unitary (up to scaling) if $|f|=1$ and if $|u_i|=1$ for all $i$
and is defined by $n+1$ real parameters 
(or by $n$ such parameters for a fixed $f$),
which we can fix or choose at random.

Now suppose that $n=2^ds$, $1\le d<k$, $d$ and $k$ are integers,
and substitute a pair of AF
matrices of recursion length $d$ for two factors $\Omega_n$ in the
above expressions. 
Then the resulting {\em abridged $f$-circulant matrix} $Z_{g,d}({\bf v})$
{\em  of recursion depth} $d$ is still unitary  (up to scaling),
defined by $n+1$ or $n$ parameters $u_i$ and $f$, 
is sparse unless the positive integer $k-d$ is  small,
and can be pre-multiplied by a vector by using $(3d+3)n$ flops.
Instead of AF matrices, we can substitute  a pair of 
ASPF, APF, ASF, AH,
ASPH, APH, or ASF
matrices
for the factors
$\Omega_n$. 
All such matrices form {\bf FAMILY (iii)} of
 $d$--{\em abridged $f$-circulant matrices}.

\begin{remark}\label{resrft}
Recall that an $n\times l$ SRFT and SRHT matrices are the products 
 $\sqrt{n/l}~D\Omega_nR$ and  $\sqrt{n/l}~DH_nR$, respectively,
 where  $H_n$ and $\Omega_n$ 
are the matrices
of (\ref{eqrfd}) and (\ref{eqdft}), $D=\diag(u_i)_{i=0}^{n-1}$, 
$u_i$ are i.i.d. variables uniformly distributed on the circle
$\{u:~|u|=\sqrt{n/l}\}$, and $R$ is the  $n\times l$ 
submatrix  formed by $l$ columns of the identity matrix $I_n$
chosen uniformly at random. Equation (\ref{eqcrcl}) shows that
we can obtain a SRFT matrix by
pre-multiplying a circulant matrix by the matrix $\Omega_n$ and 
post-multiplying it by the above matrix $R$.
\end{remark}


\subsection{Inverses of bidiagonal matrices  
}\label{sinvchh}


{\bf FAMILY (iv)}  is formed by the {\em inverses of $n\times n$ bidiagonal matrices}
$$\widehat B=(I_n+DZ)^{-1}~{\rm or}~\widehat B=(I_n+Z^TD)^{-1}$$ for 
a matrix $D$  of   primitive  type 2 and the down-shift matrix
$Z$. In particular,

 $$\widehat B=(I_n+DZ)^{-1}=\begin{pmatrix}
        ~~1  ~ & ~~0&   \dots &~\dots   ~ ~& 0 & 0\\
       ~~ b_2b_3   &~~ 1    &  0 &   & ~ 0 & ~0\\
        ~~-b_2b_3b_4  & ~~b_3b_4  &  ~~ 1 &   \ddots & ~\vdots &~ \vdots  \\
         \vdots  &  \ddots  &   ~~~  \ddots &   \ddots     &~ \vdots & ~\vdots \\
		&	&	\ddots & \ddots  &	\\
	&	&	& & 1 &0	\\
\pm b_2\cdots b_n  & \dots & \dots &-b_{n-2}b_{n-1}b_n & b_{n-1}b_n & ~1 
    \end{pmatrix}$$ if
$$I_n+DZ=\begin{pmatrix}
        1  ~ & 0&     ~  ~\dots   ~ ~& 0 & ~0\\
        -b_2   & 1    &  \ddots &    ~ 0 & ~0\\
         0 &   -b_3      & \ddots   & ~\vdots &~ \vdots  \\
         \vdots   &    &   \ddots    &~ 1 & ~0 \\
        0   & ~  \dots &      ~ ~~  \dots   ~~ ~ ~&-b_n & ~1 
    \end{pmatrix}.$$
In order to pre-multiply a matrix $\widehat B=(I_n+DZ)^{-1}$ by a vector ${\bf v}$,
however,
we do not compute its entries, but solve the linear system of equations
$(I_n+DZ){\bf x}={\bf v}$ by
using $2n-1$ flops or, in the real case, just $n-1$ additions and subtractions. 

We can randomize the matrix $\widehat B$ 
 by choosing up to $n-1$ random diagonal entries of
the matrix $D$
(whose leading entry  makes no impact on $\widehat B$).

Finally, $||\widehat B||\le \sqrt {n}$ because nonzero entries of the  lower triangular 
matrix $\widehat B=(I_n+DZ)^{-1}$ have absolute values 1,  and
clearly $||\widehat B^{-1}||=||I_n+DZ||\le \sqrt 2$. Hence 
$\kappa(\widehat B)=||\widehat B||~||\widehat B^{-1}||$
 (the spectral condition number of  $\widehat B$) cannot exceed
$\sqrt {2n}$ for $\widehat B=(I_n+DZ)^{-1}$,
and the same bound holds for $\widehat B=(I_n+Z^TD)^{-1}$.
 

\subsection{Summary of estimated numbers of flops and random variables involved}\label{sflprnd}


Table \ref{tabmlt}  shows 
upper bounds on 

(a) the numbers of random variables involved into the 
$n\times n$ matrices $\widehat B$
of the four families (i)--(iv)
and 

(b) the numbers of flops for pre-multiplication of such a matrix by
 a vector.\footnote{The asterisks in the table
show that the matrices 
of families (i) AF, (i) ASPF, and (iii) involve nonreal roots of unity.}   \\
For comparison, using a  Gaussian $n\times n$  multiplier  involves  $n^2$ random variables 
and $(2n-1)n$ flops. 

One can readily extend the estimates to $n\times l$ submatrices $B$ of the matrices
 $\widehat B$.


\begin{table}[ht] 
  \caption{The numbers of random variables and flops}
\label{tabmlt}

  \begin{center}
    \begin{tabular}{|*{8}{c|}}
      \hline
family &  (i) AH & (i) ASPH & (i) AF &  (i) ASPF & (ii) &  (iii)& (iv)  
\\ \hline
random variables & 0  &$2n$ & 0  & $2n$ & $2q+1$  
& $n$ &   $n-1$  
\\\hline
flops complex  &  $dn$ & $(d+1)n$ & $1.5dn$ & $(1.5d+1)n$  & $(2q-1)n$     & $(3d+2)n$ &   $2n-1$
 
\\\hline
 flops in real case & $dn$  & $(d+1)n$  & * & *  & $qn$ &  *    
 &  $n-1$ 
\\\hline

    \end{tabular}
  \end{center}
\end{table}

\begin{remark}\label{resprs0}
Other observations besides flop estimates  can be  decisive.
E.g., a
special recursive structure 
 of  an ARSPH matrix $H_{2^{k},d}$ and 
an ARSPF matrix $\Omega_{2^{k},d}$
allows
highly efficient  
 parallel implementation of  
their pre-multiplication by a vector based on 
Application Specific Integrated Circuits (ASICs) and 
Field-Programmable Gate Arrays (FPGAs), incorporating Butterfly
Circuits \cite{DE}.
\end{remark}


\subsection{Other basic families}\label{sbscfml}


There is a number of other interesting basic matrix families.
According to \cite[Remark 4.6]{HMT11}, ``among the structured random matrices ....
one of the strongest candidates involves sequences of random Givens rotations".
They are dense  unitary matrices
$$\frac{1}{\sqrt n}D_1G_1D_2G_2D_3\Omega_n,$$
for the DFT matrix $\Omega_n$, three random diagonal matrices
$D_1$, $D_2$ and $D_3$ of primitive type  2,
and two chains of Givens rotations  $G_1$ and $G_2$, 
each of the form
$$G(\theta_1,\dots,\theta_{n-1})=P\prod_{i=1}^{n-1}G(i,i+1,\theta_i)$$ for
a  random permutation matrix $P$,
$$G(i,i+1,\theta_i)=\diag(I_{i-1},\big (\begin{smallmatrix} c_i  & s_i\\
-s_i   & c_i\end{smallmatrix}\big ),I_{n-i-1}),~
c_i=\cos \theta_i,~
s_i=\sin \theta_i,~c_i^2+s_i^2=1.$$
Here
$\theta_1,\dots,\theta_{n-1}$ denote
$n-1$ random angles of rotation
 uniformly distributed in the range 
$0\le \phi< 2\pi$.

The DFT factor $\Omega_n$ makes the resulting matrices dense, 
but we can sparsify them
by  replacing that factor by an 
 AF, ASF, APF, or ASPF matrix having recursion depth $d<\log_2(n)$. This would also decrease 
the number of flops  involved in pre-multiplication of such a multiplier by a vector
from order $n\log_2(n)$ to $1.5dn+O(n)$. 

We can turn Givens sequences into distinct candidate families of efficient multipliers 
by replacing 
either or both of the Givens products with sparse matrices of Householder reflections
 matrices  of the form
$I_n-\frac{2{\bf h}{\bf h}^T}{{\bf h}^T{\bf h}}$
for fixed or random sparse
vectors ${\bf h}$ (cf. \cite[Section 5.1]{GL13}).

We  can  obtain a variety of efficient multiplier families 
by properly combining the matrices of basic families (i)--(iv) and 
 the above matrices. We can use just linear  combinations, but
can also apply block representation as in the following    
real  
$2\times 2$ block matrix $\frac{1}{\sqrt n}\begin{pmatrix}
Z_1({\bf u}) & Z_1({\bf v})  \\
Z_1({\bf v}) & -Z_1({\bf u}) 
\end{pmatrix}D$
for two  vectors ${\bf u}$ and ${\bf v}$
and a  matrix $D$ of primitive class 2.

We can define new matrix families by intertwining 
the Hadamard and Fourier recursive steps.

 The reader 
can find other useful families of multipliers
in our Section \ref{ststs}. E.g., according to our tests
in  Section \ref{ststs}, it turned out to be  efficient
to use  nonsingular well-conditioned (rather than unitary)
diagonal factors in the definition of some of our basic matrix families.


\section{Proof of  Theorems \ref{thall}, \ref{th0}, and \ref{th0d}}\label{sprf}


\subsection{Low-rank representation: proof}\label{slrr}


Hereafter $\mathcal R(W)$ denotes the range (column span)
of a matrix $W$.

\begin{theorem}\label{thlrnk} 
(i) For an $m\times n$ input matrix $M$ of rank $r\le n\le m$,
its rank-$r$ representation is given by the products
$R(R^TR)^{-1}R^TM=U(R)U(R)^TM$
provided that $R$ is an $n\times r$ matrix such that  $\mathcal R(R)=\mathcal R(M)$
and that $U(R)$ is a matrix obtained by means of column orthogonalization of $R$.  

(ii)  $\mathcal R(R)=\mathcal R(M)$,  
for $R=MB$ and an $n\times r$ matrix $B$, 
with probability $1$
if $B$ is Gaussian, and 

(iii) with a probability at least $1-r/|S|$ if an $n\times r$ matrix $B$ 
has i.i.d. random entries sampled 
 uniformly from a finite set $\mathcal S$
of cardinality $|S|$.
\end{theorem}
\begin{proof}
Readily verify part (i) (cf. \cite[pages 60--61]{S98}). Then note that
$\mathcal R(MB)\subseteq\mathcal R(M)$, for an $n\times r$ 
multiplier $B$.
Hence $\mathcal R(MB)=\mathcal R(M)$ if and only if $\rank(MB)=r$,
and therefore if and only if a multiplier $B$ has full rank $r$.

Now parts (ii) and (iii) follow from
 Theorem \ref{thrnd}.
\end{proof}

Parts (i) and (ii)  of Theorem \ref{thlrnk} 
imply parts (i) of Theorems \ref{th0} and  \ref{th0d}.


\subsection{Low-rank approximation: a basic step}\label{slrrlra}


Hereafter 
 $||W||_F=(\sum_{j=1}^{\rho}\sigma_j^2(W))^{1/2} 
\le \sqrt n~||W||$ 
denotes the {\em Frobenius norm} of a matrix $W$ and
 $W^+=T_{W,\rho}\Sigma_{W,\rho}^{-1}S^T_{W,\rho}$ 
denotes the  {\em Moore--Penrose pseudo
inverse} of a matrix $W$ of rank $\rho$ having compact SVD 
$S_{W,\rho}\Sigma_{W,\rho}T_{W,\rho}^T$. 
(Note that $||W^+||=\frac{1}{\sigma_{\rho}(W)}$.)

In our proofs of  Theorems  \ref{thall},
 \ref{th0}, and  \ref{th0d} we rely on the following 
lemma and theorem.


\begin{lemma}\label{letrnc} (Cf. \cite[Theorem 2.4.8]{GL13}.)
For an integer $r$ and an $m\times n$ matrix $M$ where $m\ge n>r>0$,
set to 0
the singular values $\sigma_j(M)$,  for $j>r$,
 let $M_r$ denote the resulting matrix, which is a closest rank-$r$
approximation of $M$, and write 
$M=M_r+E.$
Then
$$||E||=\sigma_{r+1}(M)~
{\rm and}~||E||_F^2=\sum_{j=r+1}^n\sigma_j^2\le \sigma_{r+1}(M)^2(n-r).$$
\end{lemma}


\begin{theorem}\label{thrrap} {\rm The error norm 
 in terms of $||(M_rB)^+||$}. 
Assume dealing with the
 matrices $M$ and $\tilde M$  
 of 
Algorithm  \ref{alg1},
$M_r$ and $E$ of Lemma \ref{letrnc}, 
  and $B\in \mathbb C^{n\times l}$ of rank $l$. 
Let $\rank(M_rB)=r$ and write $E'=EB$ and
$\Delta=||\tilde M-M||$.  Then 
\begin{equation}\label{eqe'} 
||E'||_F\le ||B||_F~||E||_F\le ||B||_F~\sigma_{r+1}(M)~ \sqrt{n-r}
\end{equation}
 and
\begin{equation}\label{eqdltsgm}
|\Delta-
\sigma_{r+1}(M)|\le
 \sqrt 8~||(M_rB)^+||~||E'||_F+O(||E'||_F^2).
\end{equation}
 \end{theorem}


\begin{proof}
Lemma \ref{letrnc} implies  bound (\ref{eqe'}). 

Next apply part (i) of
Theorem \ref{thlrnk} for matrix $M_r$ replacing $M$,
recall  that $\rank(M_rB)=l$, and obtain
$$U(M_rB)U(M_rB)^TM_r=M_r,~
\mathcal R(U(M_rB))=\mathcal R(M_rB)=\mathcal R(M_r).$$ 
Furthermore $U(M_rB)^T(M-M_r)=O_{n,n}$.
Therefore $$U(M_rB)U(M_rB)^TM=U(M_rB)U(M_rB)^TM_r=M_r.$$

Consequently,  
 $M-U(M_rB)U(M_rB)^TM=M-M_r=E$, and  so
(cf. Lemma \ref{letrnc})
\begin{equation}\label{eqrrnm0}
||M-U(M_rB)U(M_rB)^TM||=\sigma_{r+1}(M).
\end{equation}


\noindent Apply  \cite[Corollary C.1]{PQY15}, for 
$A=M_rB$  and $E$ replaced by $E'=(M-M_r)B$, and obtain 
$$||U(MB)U(MB)^T-U(M_rB)U(M_rB)^T||\le
\sqrt 8||(M_rB)^+||~||E'||_F+O(||E'||_F^2).$$

Combine this 
 bound  with
(\ref{eqrrnm0}) and
obtain (\ref{eqdltsgm}).
\end{proof}
 



By combining bounds (\ref{eqe'}) and (\ref{eqdltsgm}) obtain


\begin{equation}\label{eqdlt}
|\Delta-\sigma_{r+1}(M)|\le 
\sqrt {8(n-r)}~\sigma_{r+1}(M)~||B||_F||(M_rB)^+||+O(\sigma_{r+1}^2(M)).
\end{equation}
In our applications the value $\sqrt {8(n-r)}~\sigma_{r+1}(M)||B||_F$ is small,  
 and so the value $|\Delta-\sigma_{r+1}(M)|$ 
is  small unless the  norm $||(M_rB)^+||$ is large.
 

\subsection{Proof of Theorem \ref{thall}}\label{sprth1}


If $\rank(MB)=r_-<r$, then $\rank(\tilde M)\le r_-<r$,
$\Delta\ge \sigma_{r_-}(M)$. In this case $\Delta$ is not small because $\nrank(M)=r>r_-$,
and so Algorithm \ref{alg1}
applied to $M$ with the multiplier $B$ outputs FAILURE.

If $\rank(MB)=r>\nrank(MB)=r_-$, 
then $\rank(MB-E)=r_-<r$ for a small-norm perturbation matrix $E$.
Hence $\Delta\ge \sigma_{r_-}(M)-O(||E||)$,
and then again Algorithm \ref{alg1}
applied to $M$ with the multiplier $B$ outputs FAILURE.
This proves the ``only if" part of the claim of Theorem \ref{thall}.

Now let $\nrank(MB)=r$ and assume that we scale the matrix $B$
so that $||B||_F=1$.
Then $\rank(MB)=r$ (and so we can apply bound (\ref{eqdlt})),
and furthermore $\nrank(M_rB)=\nrank(MB)=r$.
Equation (\ref{eqdlt}) implies that 
$\Delta\approx 8\sqrt{8(n-r)}\sigma_{r+1}||(M_rB)^+||$.
Therefore $\Delta$ is a small positive value because $\nrank (M)=r$.
Thus the value $|\sigma_{r+1}|$ is small,
and part ``if" of Theorem \ref{thall} follows.


\subsection{Detailed estimates for primal and dual low-rank approximation}\label{slra}


The following theorem, proven in the next subsection,   
bounds  
the approximation errors and the 
probability of success of Algorithm \ref{alg1} for $B\in \mathcal G^{n\times l}$.
Together these bounds imply part (ii) of Theorem \ref{th0}. 

\begin{theorem}\label{th1} 
Suppose that  Algorithm \ref{alg1}  has been applied to 
an $m\times n$ matrix $M$ having numerical rank $r$ and
 that the multiplier $B$ is an $n\times l$ Gaussian matrix.


(i) Then the algorithm outputs an approximation $\tilde M$ of a matrix $M$ 
 by a rank-$l$ matrix within the  error norm bound
  $\Delta$ such that 
$|\Delta-\sigma_{r+1}(M)|\le f\sigma_{r+1}(M)/\sigma_r(M)+O(\sigma_{r+1}^2(M))$ where 
$f=\sqrt {8(n-r)}~\nu_{F,n,l}\nu^+_{r,l}$
and
$\nu_{F,n,l}$ and $\nu^+_{r,l}$
are random variables 
of Definition \ref{defnrm}.


(ii) $\mathbb E(f)<(1+\sqrt n+\sqrt l)\frac{e}{p}~\sqrt{8(n-r)rl}$, 
for  $p=l-r>0$  and  $e=2.71828\dots$.
\end{theorem}


\begin{remark}\label{reopt}
$\sigma_{r+1}(M)$ is the optimal upper bound on the norm $\Delta$, and
the expected value $\mathbb E(f)$ 
is reasonably small even for $p=1$. 
If $p=0$, then $\mathbb E(f)$  is not defined, 
 but 
the random variable $\Delta$ estimated in Theorem
\ref{th1} is still likely to be reasonably close
to $\sigma_{r+1}(M)$ (cf. part (ii) of Theorem  \ref{thsiguna}).
\end{remark}


In Section \ref{sdlrnd} we  prove the following elaboration upon dual Theorem  \ref{th0d}.

\begin{theorem}\label{th1d}
Suppose that Algorithm \ref{alg1}, applied to
  a small-norm perturbation 
of 
an $m\times n$
 fac\-tor-Gauss\-ian matrix with expected 
 rank $r<m$, uses  
an 
$n\times l$ multiplier $B$ such that $\nrank(B)=l$ and
$l\ge r$. 


(i) Then
 the algorithm outputs  
  a rank-$l$ matrix  $\tilde M$ that
 approximates the matrix $M$ within the  error norm bound
  $\Delta$ such that $|\Delta-\sigma_{r+1}(M)|\le f_d\sigma_{r+1}(M)+O(\sigma_{r+1}^2(M))$, where   
$f_d=\sqrt {8(n-r)l}~\nu^+_{r,l}\nu_{m,r}^+\kappa(B)$,
 $\kappa(B)=||B||~||B^+||$,
and  
 $\nu_{m,r}^+$ and $\nu^+_{r,l}$
are random variables 
of Definition \ref{defnrm}.


(ii) $\mathbb E(f_d)<e^2\sqrt{8(n-r)l}~\kappa(B)\frac{r}{(m-r)p}$, 
for  $p=l-r>0$  and  $e=2.71828\dots$.
\end{theorem}


\begin{remark}\label{redl}
The expected value $\mathbb E(\nu_{m,r}^+)=\frac{e\sqrt r}{m-r}$ converges to 0 as $m\rightarrow \infty$
provided that $r\ll m$. 
Consequently the expected value $\mathbb E(\Delta)=\sigma_{r+1}(M)\mathbb E(f_d)$ 
converges to the optimal value $\sigma_{r+1}(M)$
as $\frac{m}{r\sqrt {nl}}\rightarrow \infty$ provided 
that $B$ is a well-conditioned matrix
of full rank and that $1\le r<l\ll n\le m$. 
\end{remark}


\begin{remark}\label{repfr} 
 \cite[Theorem 10.8]{HMT11} also 
estimates the norm $\Delta$, but our estimate in Theorem \ref{th1}, 
in terms of random variables $\nu_{F,n,l}$ and $\nu^+_{r,l}$,
is more compact, and 
our proof is distinct and shorter than one
 in \cite{HMT11}, which involves the proofs of  \cite[Theorems 9.1, 10.4 and 10.6]{HMT11}.
\end{remark}

\begin{remark}\label{reprob}  
By virtue of 
Theorems \ref{thrnd}, $\rank(M_rB)=r$
with probability 1 if the matrix $B$ or $M$ is Gaussian,
which is the case of Theorems \ref{th1} and \ref{th1d},
and under the equation  $\rank(M_rB)=r$ we have proven bound (\ref{eqdlt}).
\end{remark}


\begin{remark}\label{rernlrpr} {\em  The Power Scheme of increasing the output accuracy 
of Algorithm \ref{alg1}.} See \cite{RST09}, \cite{HMST11}.
Define the Power Iterations
$M_i=(M^TM)^iM$, $i=1,2,\dots$. 
Then $\sigma_j(M_i)=(\sigma_j(M))^{2i+1}$
for all $i$ and $j$  \cite[equation (4.5)]{HMT11}. 
Therefore, at a reasonable computational cost, one can
dramatically decrease the ratio
$\frac{\sigma_{r+1}(M)}{\sigma_r(M)}$ and thus decrease 
 the bounds of Theorems \ref{th1}
and \ref{th1d}  accordingly.
\end{remark}

In the next two subsections
 we deduce reasonable bounds on the  norm $||(M_rB)^+||$ in both cases where 
 $M$ is a fixed  matrix and $B$ is a Gaussian matrix 
and where $B$ is fixed  matrix and $M$ is a factor Gaussian matrix 
(cf. Theorems \ref{thprm} and \ref{thdual}).
The bounds imply Theorems \ref{th1} and \ref{th1d}.


\subsection{Primal theorem: completion of the proof}\label{sprrnd}






\begin{theorem}\label{thprm} 
For $M\in \mathbb R^{m\times n}$,
 $B\in \mathcal G^{m\times l}$,
and $\nu_{r,l}^+$ 
of Definition \ref{defnrm},
it holds that
  \begin{equation}\label{eqn+}
||(M_rB)^+||\le \nu_{r,l}^+/\sigma_r(M). 
\end{equation}
\end{theorem}

\begin{proof}
Let $M_r=S_r\Sigma_rT_r^T$ be compact SVD.

By applying  Lemma \ref{lepr3}, deduce that $T_r^TB$
is a $r\times l$ Gaussian matrix. 

Denote it $G_{r,l}$ 
and obtain
$M_rB=S_r\Sigma_rT_r^TB=S_r\Sigma_rG_{r,l}$.

Write $H=\Sigma_rG_{r,l}$ and let $H=S_H\Sigma_HT_H^T$ be compact  SVD
where $S_H$ is a $r\times r$ unitary matrix.

It follows that $S=S_rS_H$ is an $m\times r$ unitary matrix.

Hence
$M_rB=S\Sigma_HT_H^T$ and $(M_rB)^+=T_H(\Sigma_H)^+S^T$
are compact SVDs of the matrices $M_rB$ and $(M_rB)^+$, respectively.

Therefore 
$||(M_rB)^+||=||(\Sigma_H)^+||=
||(\Sigma_rG_{r,l})^+||\le ||G_{r,l}^+||~||\Sigma_r^{-1}||$.

Substitute $||G_{r,l}^+||=\nu_{r,l}^+$ and $||\Sigma_r^{-1}||=1/\sigma_r(M)$
and obtain the theorem.
\end{proof}

  
Combine bounds (\ref{eqdlt}), (\ref{eqn+}),  and equation $||B||_F=\nu_{F,n,l}$
and obtain part (i) of Theorem \ref{th1}.
Combine that part with parts (ii) of Theorem \ref{thsignorm} and
(iii) of  Theorem \ref{thsiguna} 
 and obtain part (ii) of Theorem \ref{th1}.


\subsection{ Dual  theorem: completion of the proof}\label{sdlrnd}


\begin{theorem}\label{thdual}
Suppose that   
$U\in \mathbb R^{m\times r}$,
 $V\in \mathcal G^{r\times n}$, $\rank(U)=r\le \min\{m,n\}$,
$M=UV$,
and $B$ is 
a well-conditioned $n\times l$ matrix 
of full rank $l$ such that $m\ge n>l\ge r$
and $||B||_F=1$. 
Then 
 \begin{equation}\label{eqdlt1}
 ||(MB)^+|| \le ||B^+||~\nu_{r,l}^+~||U^+||.
\end{equation}
If in addition  $U\in \mathcal G^{m\times r}$, that is, if
$M$   is an $m\times n$ fac\-tor-Gauss\-ian matrix with expected rank $r$, then
 \begin{equation}\label{eqmb+}
||(MB)^+|| \le ||B^+|| \nu_{m,r}^+~\nu_{r,l}^+.
\end{equation}
\end{theorem}

\begin{proof}
Combine compact
SVDs  
$U=S_{U} \Sigma_{U}T_{U}^T$ 
and  $B=S_{B} \Sigma_{B}T_{B}^T$
and obtain $UVB=
S_{U} \Sigma_{U}T^T_{U}VS_B\Sigma_{B}T_{B}^T$.
Here $U$, $V$, $B$,  $S_U$, $\Sigma_U$,  $T_U$, $S_B$, 
$\Sigma_B$, and $T_B$ are matrices of the sizes $m\times r$, 
$r\times n$, $n\times l$, $m\times r$, $r\times r$, $r\times r$, 
$n\times l$, $l\times l$, and $l\times l$, respectively.

Now observe that $G_{r,l}=T_{U}^TVS_{B}$ is a $r\times l$
Gaussian matrix,
by  virtue of Lemma \ref{lepr3} (since $V$ is a Gaussian matrix).
Therefore $UVB=S_{U} FT_{B}^T$, for $F= \Sigma_{U}G_{r,l}\Sigma_{B}$.

Let $F=S_F\Sigma_F T_F^T$ denote compact SVD
where $\Sigma_F=\diag (\sigma_j(F))_{j=1}^r$
and $S_F$ and $T_F^T$ are unitary matrices of sizes 
$r\times r$ and $r\times l$,  respectively.

Both products $S_{U}S_F\in \mathbb R^{m\times r}$ and 
$T_F^TT_{B}^T\in \mathbb R^{r\times l}$ are unitary matrices,
 and we obtain compact SVD 
$MB=UVB=S_{MB}\Sigma_{MB}T^T_{MB}$
where $S_{MB}=S_{U}S_F$, $\Sigma_{MB}=\Sigma_F$, 
and $T^T_{MB}=T_F^TT_{B}^T$.
Therefore
  $$||(MB)^+||=||\Sigma^+_{MB}||=||\Sigma^+_F||= ||F^+||.$$
Note that $F^+=\Sigma_B^{-1}G_{r,l}^+\Sigma_U^{-1}$
because $\Sigma_B$ and $\Sigma_V$ are square nonsingular diagonal matrices.
Consequently
$$||(MB)^+||=||F^+||\le ||\Sigma_B^{-1}||~||G_{r,l}^+||~||\Sigma_U^{-1}||=||B^+||\nu_{r,l}^+||U^+||,$$
and (\ref{eqdlt1}) follows.
 \end{proof}

We also need the following result
implied by \cite[Corollary 1.4.19]{S98} for $P= -C^{-1}E$:
\begin{theorem}\label{thpert} 
Suppose $C$ and $C+E$ are two nonsingular matrices of the same size
and $$||C^{-1}E||=\theta<1.$$ Then
$$\||(C+E)^{-1}-C^{-1}||\le \frac{\theta}{1-\theta}||C^{-1}||;$$
e.g., $\||(C+E)^{-1}-C^{-1}||\le 0.5||C^{-1}||$
if $\theta\le 1/3$.
\end{theorem}

Combine (\ref{eqdlt}), (\ref{eqdlt1})
and $||B||_F\le ||B||~\sqrt l$ and  obtain
 Theorem \ref{th1d} provided that
$M$ is a fac\-tor-Gauss\-ian matrix $UV$ with expected rank $r$.
Apply Theorem \ref{thpert} to extend the results 
to the case where  
$M=UV+E$ and the norm $||E||$ is small, completing the proof of 
Theorem \ref{th1d}.


\begin{remark}\label{regnrl} 
If $U\in \mathcal G^{m\times r}$, for $m-r\ge 4$, then
it is likely that
$\nrank(U)=r$ by virtue of Theorem \ref{thsiguna},
and our proof of bound (\ref{eqdlt1}) applies even if  we assume that
$\nrank(U)=r$ rather than $U\in \mathcal G^{m\times r}$.
\end{remark}


\section{Numerical Tests}\label{ststs}


 Numerical experiments have been 
 performed by  Xiaodong Yan for Tables  
\ref{tab67}--\ref{tab614}
and by John Svadlenka and Liang Zhao for the other tables.
The tests have been run by using MATLAB  
 in the Graduate Center of the City University of New York 
on a Dell computer with the Intel Core 2 2.50 GHz processor and 4G memory running 
Windows 7; 
in particular the standard normal distribution function randn of MATLAB
has been applied in order to generate Gaussian matrices.

We calculated the $\xi$-rank, i.e., the number of singular values 
exceeding $\xi$, by applying the MATLAB function "svd()". 
We have set $\xi=10^{-5}$ in Sections \ref{ststssvd} and
 \ref{ststslo} and  $\xi=10^{-6}$
in Section \ref{s17m}.


\subsection{Tests for inputs generated via SVD}\label{ststssvd}


In the tests of this subsection we generated $n\times n$ input matrices $M$  
 by extending the customary recipes of [H02, Section 28.3]. Namely, we first  
generated matrices $S_M$ and $T_M$ 
by means of the orthogonalization of  
$n\times n$ Gaussian matrices. Then we defined
$n\times n$ matrices $M$ by 
their compact SVDs, $M=S_M\Sigma_M T_M^T$,
for $\Sigma_M=\diag(\sigma_j)_{j=1}^n$; 
 $\sigma_j=1/j,~j=1,\dots,r$,
$\sigma_j=10^{-10},~j=r+1,\dots,n$, 
 and  $n=256,
512,
1024$.
(Hence $||M||=1$ and 
$\kappa(M)=||M||~||M^{-1}||=10^{10}$.) 

Table \ref{tab1} shows
the average output error norms $\Delta$ 
over  1000 tests of  Algorithm \ref{alg1} applied to 
these matrices $M$ for 
each pair of $n$ and $r$, 
  $n=256,
512,
1024$, $r=8,32$, and
each of the following three groups of multipliers:
3-AH multipliers, 
3-ASPH  multipliers, both
 defined by 
 Hadamard recursion (\ref{eqfd}),  
 for $d=3$, and 
dense  multipliers $B=B(\pm 1,0)$  
having i.i.d. entries $\pm 1$ and 0,
each value chosen with probability 1/3.


\begin{table}[ht] 
  \caption{Error norms for SVD-generated inputs 
and 3-AH, 3-ASPH,  and $B(\pm 1,0)$ multipliers}
\label{tab1}

  \begin{center}
    \begin{tabular}{|*{8}{c|}}
      \hline
$n$ & $r$  &3-AH&3-ASPH& $B(\pm 1,0)$ 
\\ \hline
256 & 8 & 2.25e-08 & 2.70e-08 & 2.52e-08 
\\\hline
256 & 32 &5.95e-08 & 1.47e-07 & 3.19e-08 
\\\hline
512 & 8 &4.80e-08 & 2.22e-07 & 4.76e-08 
\\\hline
512 & 32 & 6.22e-08 & 8.91e-08 & 6.39e-08 
\\\hline
1024 & 8 & 5.65e-08 & 2.86e-08 & 1.25e-08
\\\hline
1024 & 32 & 1.94e-07 & 5.33e-08 & 4.72e-08
\\\hline
    \end{tabular}
  \end{center}
\end{table}


Tables \ref{tab67}--\ref{tab614} show  the  mean and maximal values 
 of such an error norm in the case of   
(a) real Gaussian multipliers $B$ and dense real Gaussian 
subcirculant multipliers $B$, for $q=n$, each  defined by its first column filled with 
either (b) i.i.d.  Gaussian
variables or (c) random variables
  $\pm 1$. 
Here 
and hereafter in this section we assigned each random 
signs $+$ or $-$ with probability 0.5.


\begin{table}[ht] 
  \caption{Error norms for SVD-generated inputs and
Gaussian multipliers}
\label{tab67}
  \begin{center}
    \begin{tabular}{| c |  c | c |  c |c|}
      \hline
$r$   & $n$ & \bf{mean} & \bf{max} \\ \hline	
8 & 256 & $ 7.54\times 10^{-8} $  &  $ 1.75\times 10^{-5} $    \\ \hline		
8 & 512 &  $ 4.57\times 10^{-8} $  &  $ 5.88\times 10^{-6} $  \\ \hline		
8 & 1024 &  $ 1.03\times 10^{-7} $  &  $ 3.93\times 10^{-5} $    \\ \hline		
32 & 256 & $ 5.41\times 10^{-8} $  &  $ 3.52\times 10^{-6} $     \\ \hline
32 & 512 & $ 1.75\times 10^{-7} $  &  $ 5.57\times 10^{-5} $   \\ \hline		
32 & 1024 & $ 1.79\times 10^{-7} $  &  $ 3.36\times 10^{-5} $    \\ \hline	
    \end{tabular}
  \end{center}
\end{table}


\begin{table}[ht] 
  \caption{Error norms  for SVD-generated inputs  and
Gaussian  subcirculant multipliers}
\label{tab611}
  \begin{center}
    \begin{tabular}{| c | c | c | c | c | c |c|}     \hline
$r$  &  $n$ & \bf{mean} & \bf{max}  \\ \hline
8 & 256 & $ 3.24\times 10^{-8} $  &  $ 2.66\times 10^{-6} $   \\ \hline		
8 & 512 &  $ 5.58\times 10^{-8} $  &  $ 1.14\times 10^{-5} $  \\ \hline		
8 & 1024 &  $ 1.03\times 10^{-7} $  &  $ 1.22\times 10^{-5} $  \\ \hline		
32 & 256 & $ 1.12\times 10^{-7} $  &  $ 3.42\times 10^{-5} $   \\ \hline
32 & 512 & $ 1.38\times 10^{-7} $  &  $ 3.87\times 10^{-5} $   \\ \hline		
32 & 1024 & $ 1.18\times 10^{-7} $  &  $ 1.84\times 10^{-5} $  \\ \hline	
    \end{tabular}
  \end{center}
\end{table}


\begin{table}[ht] 
  \caption{Error norms for SVD-generated inputs and
random subcirculant
  multipliers  filled with $\pm 1$}
\label{tab614}
  \begin{center}
    \begin{tabular}{| c | c | c | c | c | c |c|}     \hline
$r$  &  $n$ & \bf{mean} & \bf{max} \\ \hline			
8 & 256 & $ 7.70\times 10^{-9} $  &  $ 2.21\times 10^{-7} $   \\ \hline		
8 & 512 &  $ 1.10\times 10^{-8} $  &  $ 2.21\times 10^{-7} $  \\ \hline		
8 & 1024 &  $ 1.69\times 10^{-8} $  &  $ 4.15\times 10^{-7} $   \\ \hline		
32 & 256 & $ 1.51\times 10^{-8} $  &  $ 3.05\times 10^{-7} $   \\ \hline
32 & 512 & $ 2.11\times 10^{-8} $  &  $ 3.60\times 10^{-7} $   \\ \hline		
32 & 1024 & $ 3.21\times 10^{-8} $  &  $ 5.61\times 10^{-7} $  \\ \hline		
    \end{tabular}
  \end{center}
\end{table}

 

Table \ref{LowRkEx} displays  the
average error norms in
the case of multipliers $B$
of eight kinds defined below, all  
generated from the following  Basic Sets 1, 2 and 3
of $n\times n$ multipliers:

{\em Basic Set 1}:  3-APF
multipliers defined by 
three Fourier recursive steps of
 equation (\ref{eqfd}), for $d=3$,
with no scaling, but  
with a random column permutation.

{\em Basic Set 2}: Sparse real circulant matrices $Z_1({}\bf v)$
of family (ii) of  Section \ref{scrcsp} (for $q=10$) 
 having the first column vectors  ${\bf v}$ filled with zeros,
except for  ten random coordinates filled with random integers $\pm 1$.

{\em Basic Set 3}:   Sum of two scaled inverse bidiagonal matrices. 
We first filled the main diagonals of both matrices with the integer 101
 and  their first subdiagonals 
 with $\pm 1$. Then
we  multiplied  each matrix by a  diagonal 
matrix $\diag(\pm 2^{b_i})$, where $b_i$ were random integers
uniformly chosen from 0 to 3.

For multipliers $B$ we used the $n\times r$  western 
(leftmost) blocks of $n\times n$ matrices
of the following classes:
\begin{enumerate}
\item
  a matrix from Basic Set 1; 
\item
  a matrix from Basic Set 2;
\item
 a matrix from Basic Set 3;
\item
 the product of two matrices of Basic Set 1;
\item
 the product of two matrices of Basic Set 2;
\item
 the product of two matrices of Basic Set 3;
\item
 the sum of two matrices of Basic Sets 1 and 3,
and 
\item
 the sum of two matrices of Basic Sets 2 and 3.
\end{enumerate}
The tests
produced the results similar to the ones of Tables \ref{tab1}--\ref{tab614}.
 
In sum, for all classes of input  pairs $M$ and $B$ and all pairs of integers $n$ and $r$,
Algorithm \ref{alg1} with our preprocessing 
has consistently output approximations to rank-$r$ input matrices with  the
average error norms 
 ranged from $10^{-7}$ or $10^{-8}$ to about $10^{-9}$
in all our tests.

\begin{table}[ht] 
  \caption{Error norms  for SVD-generated inputs  and
multipliers of eight classes}
\label{LowRkEx}
  \begin{center}
    \begin{tabular}{| c |  c | c |  c |c|c|c|c|c|c|}
      \hline
$n$ & $r$ &class 1 &class 2 &class 3 &class 4 &class 5 &class 6 &class 7 &class 8 \\\hline
256 & 8 &5.94e-09 &4.35e-08 &2.64e-08 &2.20e-08 &7.73e-07 &5.15e-09 &4.08e-09 &2.10e-09  \\\hline
256 & 32 &2.40e-08 &2.55e-09 &8.23e-08 &1.58e-08 &4.58e-09 &1.36e-08 &2.26e-09 &8.83e-09 \\\hline 
512 & 8 &1.11e-08 &8.01e-09 &2.36e-09 &7.48e-09 &1.53e-08 &8.15e-09 &1.39e-08 &3.86e-09  \\\hline
512 & 32 &1.61e-08 &4.81e-09 &1.61e-08 &2.83e-09 &2.35e-08 &3.48e-08 &2.25e-08 &1.67e-08\\\hline 
1024 & 8 &5.40e-09 &3.44e-09 &6.82e-08 &4.39e-08 &1.20e-08 &4.44e-09 &2.68e-09 &4.30e-09 \\\hline 
1024 & 32 &2.18e-08 &2.03e-08 &8.72e-08 &2.77e-08 &3.15e-08 &7.99e-09 &9.64e-09 &1.49e-08\\\hline 
    \end{tabular}
  \end{center}
\end{table}


  We summarize the  results of the tests of this subsection for $n=1024$ and $r=8,32$
in Figure \ref{LowRkTest1}.

\begin{figure}[htb] 
\centering
\includegraphics[scale=0.7] {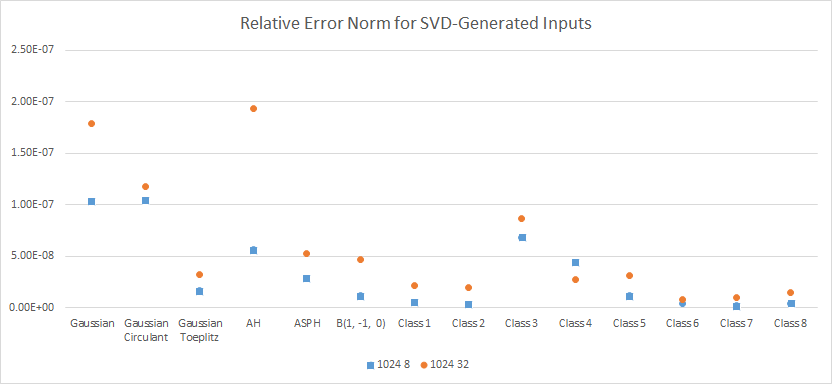}
\caption{Error norms in the tests of Section \ref{ststssvd}}
\label{LowRkTest1}
\end{figure}

 

 


\subsection{Tests for inputs generated via the discretization of a Laplacian operator
and via the approximation of an inverse finite-difference operator}\label{ststslo}


Next we present the test results for Algorithm \ref{alg1} applied
 to input matrices for computational problems of two kinds,
both replicated from  \cite{HMT11}, namely, the matrices of
 
(i) the discretized single-layer Laplacian operator and 

(ii) the approximation of the inverse of a finite-difference operator.

{\em Input matrices (i).} We considered the Laplacian operator
$[S\sigma](x) = c\int_{\Gamma_1}\log{|x-y|}\sigma(y)dy,x\in\Gamma_2$,
from  \cite[Section 7.1]{HMT11},
for two contours $\Gamma_1 = C(0,1)$ and $\Gamma_2 = C(0,2)$  on the complex plane.
Its dscretization defines an $n\times n$ matrix $M=(m_{ij})_{i,j=1}^n$  
where
$m_{i,j} = c\int_{\Gamma_{1,j}}\log|2\omega^i-y|dy$
for a constant $c$ such that  $||M||=1$ and
for the arc $\Gamma_{1,j}$  of the contour $\Gamma_1$ defined by
the angles in the range $[\frac{2j\pi}{n},\frac{2(j+1)\pi}{n}]$.

We applied Algorithm \ref{alg1}
supported by  
 three iterations of the Power Scheme of Remark  \ref{rernlrpr} 
and used with multipliers
 $B$ being the $n\times r$ leftmost submatrices of $n\times n$
matrices of 
the following five  classes: 
\begin{itemize}
\item
 Gaussian  multipliers, 
\item
 Gaussian Toeplitz  multipliers $T=(t_{i-j})_{i=0}^{n-1}$ 
for i.i.d. Gaussian variables $t_{1-n},\dots,t_{-1}$,$t_0,t_1,\dots,t_{n-1}$.
  \item
 Gaussian circulant  multipliers $\sum_{i=0}^{n-1}v_iZ_1^i$, 
for i.i.d. Gaussian variables $v_0,\dots,v_{n-1}$ and the unit circular matrix $Z_1$
of Section \ref{sdfcnd}.  
\item
 Abridged permuted  Fourier (3-APF) multipliers, and 
\item
 Abridged permuted Hadamard (3-APH) multipliers.
\end{itemize}

As in the previous subsection,
we defined each 
 3-APF and 3-APH  matrix by applying
three recursive steps of equation (\ref{eqfd}) followed
by a single random column permutation.

 
We applied Algorithm \ref{alg1} with  multipliers of all five listed classes.
For each setting we repeated the test 1000 times and calculated the mean and standard deviation of the error norm $||\tilde M - M||$. 
  
\begin{table}[h]
\caption{Low-rank approximation  of Laplacian  matrices}
\label{ExpHMT1}
\begin{center}
\begin{tabular}{|*{5}{c|}}
\hline
$n$ 	& multiplier 	& $r$	& mean	& std\\\hline
200 & Gaussian &  3.00 & 1.58e-05 & 1.24e-05\\\hline
200 & Toeplitz &  3.00 & 1.83e-05 & 7.05e-06\\\hline
200 & Circulant &  3.00 & 3.14e-05 & 2.30e-05\\\hline
200 & 3-APF &  3.00 & 8.50e-06 & 5.15e-15\\\hline
200 & 3-APH &  3.00 & 2.18e-05 & 6.48e-14\\\hline
400 & Gaussian &  3.00 & 1.53e-05 & 1.37e-06\\\hline
400 & Toeplitz &  3.00 & 1.82e-05 & 1.59e-05\\\hline
400 & Circulant &  3.00 & 4.37e-05 & 3.94e-05\\\hline
400 & 3-APF &  3.00 & 8.33e-06 & 1.02e-14\\\hline
400 & 3-APH &  3.00 & 2.18e-05 & 9.08e-14\\\hline
2000 & Gaussian &  3.00 & 2.10e-05 & 2.28e-05\\\hline
2000 & Toeplitz &  3.00 & 2.02e-05 & 1.42e-05\\\hline
2000 & Circulant &  3.00 & 6.23e-05 & 7.62e-05\\\hline
2000 & 3-APF &  3.00 & 1.31e-05 & 6.16e-14\\\hline
2000 & 3-APH &  3.00 & 2.11e-05 & 4.49e-12\\\hline
4000 & Gaussian &  3.00 & 2.18e-05 & 3.17e-05\\\hline
4000 & Toeplitz &  3.00 & 2.52e-05 & 3.64e-05\\\hline
4000 & Circulant &  3.00 & 8.98e-05 & 8.27e-05\\\hline
4000 & 3-APF &  3.00 & 5.69e-05 & 1.28e-13\\\hline
4000 & 3-APH &  3.00 & 3.17e-05 & 8.64e-12\\\hline
\end{tabular}
\end{center}
\end{table}

{\em Input matrices (ii).} We similarly applied Algorithm \ref{alg1} to the input matrix $M$ 
being the inverse of a large sparse matrix obtained from a finite-difference operator
of  \cite[Section 7.2]{HMT11} 
and observed  similar results
with all structured  and Gaussian multipliers.

We performed 1000 tests for every class of pairs of $n\times n$ or $m\times n$ matrices 
of classes (i) or (ii), respectively,
and 
$n\times r$ multipliers for every fixed triple of $m$, $n$, and $r$ or pair of $n$ and $r$.

Tables \ref{ExpHMT1} and \ref{ExpHMT2} display the resulting data for the mean values and standard deviation of the error norms, and we summarize the results of the tests of this subsection 
in Figure \ref{LowRkTest2}.

\begin{table}[h]
\caption{Low-rank approximation of the matrices of discretized finite-difference operator}
\label{ExpHMT2}
\begin{center}
\begin{tabular}{|*{6}{c|}}
\hline
$m$ 	& $n$ 	& multiplier 	& $r$	& mean	& std\\\hline
88 & 160 & Gaussian &  5.00 & 1.53e-05 & 1.03e-05\\\hline
88 & 160 & Toeplitz &  5.00 & 1.37e-05 & 1.17e-05\\\hline
88 & 160 & Circulant &  5.00 & 2.79e-05 & 2.33e-05\\\hline
88 & 160 & 3-APF &  5.00 & 4.84e-04 & 2.94e-14\\\hline
88 & 160 & 3-APH &  5.00 & 4.84e-04 & 5.76e-14\\\hline
208 & 400 & Gaussian & 43.00 & 4.02e-05 & 1.05e-05\\\hline
208 & 400 & Toeplitz & 43.00 & 8.19e-05 & 1.63e-05\\\hline
208 & 400 & Circulant & 43.00 & 8.72e-05 & 2.09e-05\\\hline
208 & 400 & 3-APF & 43.00 & 1.24e-04 & 2.40e-13\\\hline
208 & 400 & 3-APH & 43.00 & 1.29e-04 & 4.62e-13\\\hline
408 & 800 & Gaussian & 64.00 & 6.09e-05 & 1.75e-05\\\hline
408 & 800 & Toeplitz & 64.00 & 1.07e-04 & 2.67e-05\\\hline
408 & 800 & Circulant & 64.00 & 1.04e-04 & 2.67e-05\\\hline
408 & 800 & 3-APF & 64.00 & 1.84e-04 & 6.42e-12\\\hline
408 & 800 & 3-APH & 64.00 & 1.38e-04 & 8.65e-12\\\hline
\end{tabular}
\end{center}
\end{table} 


\begin{figure}[htb] 
\centering
\includegraphics[scale=0.7] {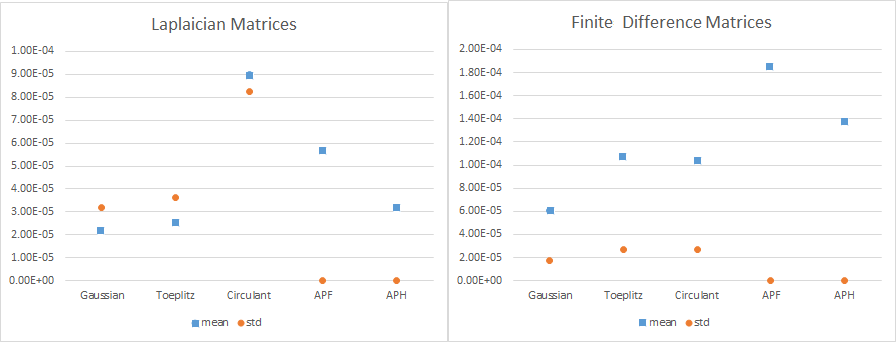}
\caption{Error norms in the tests of Section \ref{ststslo}}
\label{LowRkTest2}
\end{figure}


\subsection{Tests with additional classes of multipliers}\label{s17m} 

In this subsection we display the mean values and standard deviations
of the  error norms observed 
when we repeated the tests of the two previous subsections 
for the same three classes of input matrices
 (that is, SVD-generated, Laplacian, and matrices obtained by discretization of 
 finite difference operators), but now we applied Algorithm \ref{alg1} with
  seventeen  additional classes of multipliers (besides its control application with
 Gaussian multipliers). 
 
We tested  Algorithm \ref{alg1} applied to $1024\times 1024$ SVD-generated input matrices having numerical nullity $r = 32$, to $400 \times 400$ Laplacian input matrices
having numerical nullity $r = 3$, 
and
to $408 \times 800$ matrices having numerical nullity $r = 64$ and
representing finite-difference inputs. 

Then again we repeated the tests 1000 times for each class of input matrices and each 
size of an input and a multiplier, and we display the resulting average error norms 
in Table \ref{SuperfastTable} and Figures \ref{SuperfastSVD}--\ref{SuperfastFD}.

We used multipliers defined as the seventeen sums of $n\times r$ matrices
of the following basic families:

\begin{itemize}
  \item
  3-ASPH matrices
\item
  3-APH matrices
\item
  Inverses of bidiagonal matrices 
\item
 Random permutation matrices
\end{itemize}

Here every 3-APH matrix has been defined by three Hadamard's recursive steps
(\ref{eqrfd}) followed by random permutation.
Every 3-ASPH matrix has been defined similarly, but also random scaling has 
been applied with a diagonal matrix $D=\diag(d_i)_{i=1}^n$
having the values of random i.i.d. variables $d_i$ uniformly chosen from the set
$\{1/4,1/2,1,2,4\}$.
 
We permuted all inverses of bidiagonal matrices except for Class 5 of multipliers.

Describing our multipliers we use the following acronyms and abbreviations:
``IBD" for ``the inverse of a bidiagonal",
``MD" for ``the main diagonal", ``SB" for ``subdiagonal", and ``SP" for ``superdiagonal".
We write ``MD$i$", ``$k$th SB$i$" and ``$k$th SP$i$" in order to denote
that the main diagonal, the $k$th subdiagonal, or  the $k$th superdiagonal 
of a bidiagonal matrix, respectively,
was filled with the integer $i$.

\begin{itemize}

\item Class 0:	Gaussian
\item Class 1:	Sum of a 3-ASPH  and two IBD matrices: \\
	B1 with MD$-1$  and  2nd SB$-1$ and
	B2 with MD$+1$ and 1st SP$+1$
\item Class 2:	Sum of a 3-ASPH  and two IBD matrices: \\
        B1 with MD$+1$ and 2nd SB$-1$  and  
        B2 with MD$+1$ and 1st SP$-1$
\item Class 3:	Sum of   a 3-ASPH  and two IBD matrices: \\
	B1 with MD$+1$ and  1st SB$-1$ and
	B2 with MD $+1$  and 1st SP$-1$ 
\item Class 4:	Sum of  a 3-ASPH  and two IBD matrices: \\
	 B1 with MD$+1$ and 1st SB$+1$ and
        B2 with MD$+1$ and 1st SP$-1$
\item Class 5:	Sum of  a 3-ASPH  and two IBD matrices:  \\
	B1 with MD$+1$ and 1st SB$+1$ and B2 with MD$+1$ and 1st SP$-1$
\item Class 6:	Sum of a 3-ASPH  and three IBD matrices:\\
	B1 with MD$-1$ and  2nd SB$-1$,
	B2 with MD$+1$ and 1st SP$+1$ and
	B3 with MD$+1$ and 9th SB$+1$
\item Class 7:	Sum of a 3-ASPH  and three IBD matrices:\\
	 B1 with  MD$+1$ and  2nd SB$-1$, 
         B2 with MD$+1$ and 1st SP$-1$, and
	B3 with  MD$+1$ and  8th SP$+1$
\item Class 8:	Sum of a 3-ASPH  and three IBD matrices:\\
	B1 with   MD$+1$ and 1st SB$-1$,
	B2 with   MD$+1$ and 1st SP$-1$, and
	B3 with   MD$+1$ and 4th  SB$+1$
\item Class 9:	Sum of a 3-ASPH  and three IBD matrices:\\
	B1 with   MD$+1$ and 1st SB$+1$,
	B2 with   MD$+1$ and 1st SP$-1$, and
	B3 with   MD$-1$ and 3rd SP$+1$
\item Class 10:	Sum of three IBD matrices:\\
	B1 with   MD$+1$ and 1st SB$+1$,
	 B2 with   MD$+1$ and 1st SP$-1$, and
	 B3 with   MD$-1$ and 3rd SP$+1$
\item Class 11:	Sum of a 3-APH  and three IBD matrices:\\
	 B1 with   MD$+1$ and  2nd SB$-1$,
	 B2 with   MD$+1$ and 1st SP$-1$, and
	 B3 with   MD$+1$ and  8th SP$+1$
\item Class 12:	Sum of a 3-APH  and two IBD matrices:\\
	 B1 with   MD$+1$ and 1st SB$-1$ and
	 B2 with   MD$+1$ and 1st SP$-1$
\item Class 13:	Sum of a 3-ASPH  and a permutation matrix
\item Class 14:	Sum of a 3-ASPH  and two permutation matrices
\item Class 15:	Sum of a 3-ASPH  and three permutation matrices
\item Class 16:	Sum of a 3-APH  and three  permutation matrices
\item Class 17:	Sum of a 3-APH  and two permutation matrices
\end{itemize}

\begin{figure}[htb] 
\centering
\includegraphics[scale=0.7] {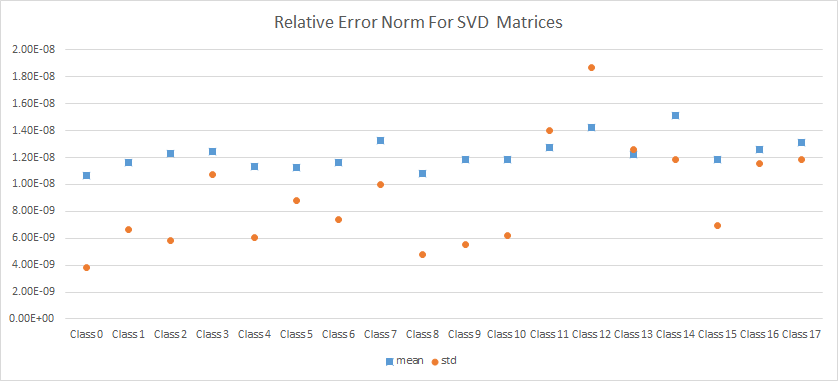}
\caption{Relative Error Norm for SVD-generated Input Matrices}
\label{SuperfastSVD}
\end{figure}

\begin{figure}[htb] 
\centering
\includegraphics[scale=0.7] {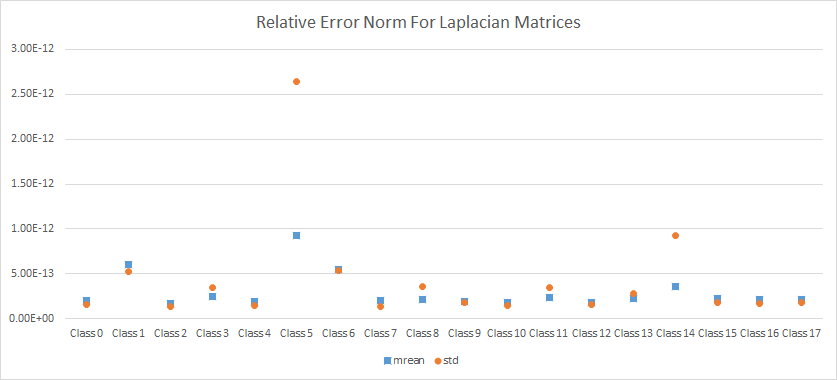}
\caption{Relative Error Norm for Lapacian Input  Matrices}
\label{SuperfastLP}
\end{figure}

\begin{figure}[htb] 
\centering
5\includegraphics[scale=0.7] {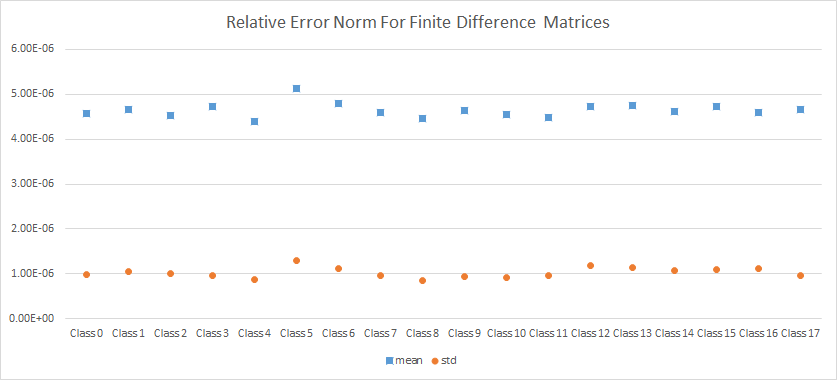}
\caption{Relative Error Norm for Finite-Difference Input Matrices}
\label{SuperfastFD}
\end{figure}

\begin{table}[htb] \label{SuperfastTable}
\begin{center}
\begin{tabular}{|c|c|c|c|c|c|c|}
\hline
			& \multicolumn{2}{|c|}{SVD-generated Matrices} & \multicolumn{2}{|c|}{Laplacian Matrices} & \multicolumn{2}{|c|}{Finite Difference Matrices}\\\hline
 \text{Class No.} & \text{Mean} & \text{Std} & \text{Mean} & \text{Std} & \text{Mean} & \text{Std} \\\hline
Class 0	&	3.54E-09	&	3.28E-09	&	4.10E-14	&	2.43E-13	&	1.61E-06	&	1.35E-06\\\hline
Class 0	&	1.07E-08	&	3.82E-09	&	2.05E-13	&	1.62E-13	&	4.58E-06	&	9.93E-07\\\hline
Class 1	&	1.16E-08	&	6.62E-09	&	6.07E-13	&	5.20E-13	&	4.67E-06	&	1.04E-06\\\hline
Class 2	&	1.23E-08	&	5.84E-09	&	1.69E-13	&	1.34E-13	&	4.52E-06	&	1.01E-06\\\hline
Class 3	&	1.25E-08	&	1.07E-08	&	2.46E-13	&	3.44E-13	&	4.72E-06	&	9.52E-07\\\hline
Class 4	&	1.13E-08	&	6.09E-09	&	1.93E-13	&	1.48E-13	&	4.38E-06	&	8.64E-07\\\hline
Class 5	&	1.12E-08	&	8.79E-09	&	9.25E-13	&	2.64E-12	&	5.12E-06	&	1.29E-06\\\hline
Class 6	&	1.16E-08	&	7.42E-09	&	5.51E-13	&	5.35E-13	&	4.79E-06	&	1.12E-06\\\hline
Class 7	&	1.33E-08	&	1.00E-08	&	1.98E-13	&	1.30E-13	&	4.60E-06	&	9.52E-07\\\hline
Class 8	&	1.08E-08	&	4.81E-09	&	2.09E-13	&	3.60E-13	&	4.47E-06	&	8.57E-07\\\hline
Class 9	&	1.18E-08	&	5.51E-09	&	1.87E-13	&	1.77E-13	&	4.63E-06	&	9.28E-07\\\hline
Class 10	&	1.18E-08	&	6.23E-09	&	1.78E-13	&	1.42E-13	&	4.55E-06	&	9.08E-07\\\hline
Class 11	&	1.28E-08	&	1.40E-08	&	2.33E-13	&	3.44E-13	&	4.49E-06	&	9.67E-07\\\hline
Class 12	&	1.43E-08	&	1.87E-08	&	1.78E-13	&	1.61E-13	&	4.74E-06	&	1.19E-06\\\hline
Class 13	&	1.22E-08	&	1.26E-08	&	2.21E-13	&	2.83E-13	&	4.75E-06	&	1.14E-06\\\hline
Class 14	&	1.51E-08	&	1.18E-08	&	3.57E-13	&	9.27E-13	&	4.61E-06	&	1.08E-06\\\hline
Class 15	&	1.19E-08	&	6.93E-09	&	2.24E-13	&	1.76E-13	&	4.74E-06	&	1.09E-06\\\hline
Class 16	&	1.26E-08	&	1.16E-08	&	2.15E-13	&	1.70E-13	&	4.59E-06	&	1.12E-06\\\hline
Class 17	&	1.31E-08	&	1.18E-08	&	1.25E-14	&	5.16E-14	&	1.83E-06	&	1.55E-06\\\hline

\end{tabular}
\caption{Relative Error Norm in Tests with Multipliers of  Additional Classes}
\end{center}
\end{table}

The tests show quite accurate outputs even where we applied Algorithm \ref{alg1}
with very sparse multipliers of classes 13--17.


\clearpage

 
\section{Conclusions: Sample Extension to 
 the Computations for Least Squares Regression (LSR)}\label{sext}


Our duality techniques for  the average inputs 
can  be extended to the acceleration of various 
matrix computations involving random multipliers.
In this concluding section we describe such a sample  
 extension to the following fundamental problem of matrix computations  (cf. \cite{GL13}).

\begin{problem}\label{pr1} {\em Least Squares Solution of an Overdetermined Linear System of Equations.}
Given two integers $m$ and $d$ such that $1\le d<m$,
a matrix $A\in \mathbb R^{m\times d}$, and a vector ${\bf b}\in \mathbb R^{m}$,
compute a vector ${\bf x}$ that minimizes the norm $||A{\bf x}-{\bf b}||$.
\end{problem}

If a matrix $A$ has full rank $n$, then unique solution 
is given by the vector ${\bf x}=(A^TA)^{-1}A^T{\bf b}$,
satisfying the linear system of normal equations $A^TA{\bf x}=A^T{\bf b}$.
Otherwise solution is not unique, and a solution  ${\bf x}$ having the
minimum norm is given by the vector $A^+{\bf b}$. 
In the important case where $m\gg d$ and an approximate solution
is acceptable,  Sarl\'os in \cite{S06} proposed  to
  simplify the computations as follows:

 
\begin{algorithm}\label{algapprls} 
{\rm  Least Squares Regression (LSR).}

 
\begin{description}





\item[{\sc Initialization:}] 
 Fix an 
 integer $k$ such that $1\le k\ll m$. 


\item[{\sc Computations:}]

\begin{enumerate}
\item 
Generate a scaled $k\times m$ Gaussian matrix $F$.

\item 
Compute the matrix $FA$ and the vector $F{\bf b}$.
\item 
Output
a solution $\tilde {\bf x}$ to the compressed Problem \ref{pr1}
where the matrix $A$ and the vector ${\bf b}$
are replaced by the matrix 
$FA$ and the vector $F{\bf b}$,
respectively.
\end{enumerate}


\end{description}


\end{algorithm}


Now write $M=(A~|~{\bf b})$ and ${\bf y}=\begin{pmatrix}
{\bf x} \\ -1
\end{pmatrix}$
and compare the error norms 
$||FM\tilde{\bf y}||=||FA\tilde{\bf x}-F{\bf b}||$ (of the output $\tilde{\bf x}$ 
of the latter algorithm) and $||M{\bf y}||=||A{\bf x}-{\bf b}||$
(of the solution ${\bf x}$ of the original Problem \ref{pr1}).

 
\begin{theorem}\label{thlsrd}  {\em  \cite[Theorem 2.3]{W14}.}
Suppose that we are given 
two tolerance values $\delta$ and  $\xi$, $0<\delta<1$ and  $0<\xi<1$, 
three integers $k$, $m$ and  $d$ such that $1\le d<m$
and $$k=(d+\log(1/\delta)\xi^{-2})\theta,$$ for a certain constant $\theta$, and
a matrix $G_{k,m}\in \mathcal G^{k\times m}$. 
Then, with a probability at least $1-\delta$, it holds that 
$$(1-\xi)||M{\bf y}||\le \frac{1}{\sqrt k}||G_{k,m}M{\bf y}||\le (1+\xi)||M{\bf y}||$$
for all matrices $M\in \mathbb R^{m\times (d+1)}$
and all vectors ${\bf y}=(y_i)_{i=0}^{d}\in \mathbb R^{d+1}$ normalized so that $y_d=-1$.
\end{theorem}

 The
theorem implies that with a probability at least $1-\delta$,
Algorithm \ref{algapprls} outputs an approximate solution to Problem \ref{pr1}
within the error norm bound $\xi$
provided that $k=(d+\log(1/\delta)\xi^{-2})\theta$
and $F=\frac{1}{\sqrt k}G_{k,m}$.\footnote{Such approximate solutions 
serve as preprocessors for practical implementation of
numerical linear algebra algorithms 
for Problem \ref{pr1} of least squares computation \cite[Section 4.5]{M11}, \cite{RT08}, \cite{AMT10}.}

 For $m\gg k$, the computational cost of performing the algorithm for approximate 
solution  
dramatically decreases versus the cost  of computing exact solution,
but can still be prohibitively high at the stage of computing the matrix product
$FM$ for $F=G_{k,m}$. 
In a number of papers the former cost estimate has been dramatically 
decreased by means of replacing a multiplier 
$F=\frac{1}{\sqrt k}G_{k,m}$ with various random sparse and structured matrices
(see \cite[Section 2.1]{W14}), for which the bound of Theorem \ref{thlsrd}
still holds for all matrices
$M\in \mathbb R^{m\times (d+1)}$, although
at the expense of increasing significantly the dimension $k$.  

Can we achieve similar progress without such an increase?
\cite{CW13} provides positive probabilistic answer
based on the Johnson--Lindenstrauss Theorem,
while the
  following theorem does this by using our duality approach
 in the case where $M$ is  the average matrix in $\mathbb R^{m\times (d+1)}$
under the Gaussian probability distribution:

 
\begin{theorem}\label{thlsrdd} {\em Dual LSR.}
The bound 
of Theorem \ref{thlsrd} holds with a probability 
at least $1-\delta$ where $\sqrt k~ M\in \mathcal G^{m\times (d+1)}$
and $F\in \mathbb R^{k\times m}$ is an orthogonal matrix.
\end{theorem}
\begin{proof} 
Theorem \ref{thlsrd} has been proven in \cite[Section 2]{W14} in
the case where  
$\sqrt k ~FM\in \mathcal G^{k\times (d+1)}$.
 This is the case where $\sqrt k ~F\in  \mathcal G^{k\times m}$
and $M$ is an orthogonal $m\times (d+1)$ matrix,
but  is also the case under the assumptions of  
Theorem \ref{thlsrdd}, by virtue of Lemma \ref{lepr3}. 
\end{proof}

Theorem \ref{thlsrdd} supports the computation of
an approximate randomized solution of
LSR Problem \ref{pr1}
for any orthogonal multiplier $F$
(e.g., an abridged  scaled Hadamard's multiplier or a count sketch
multiplier)
and for an input matrix $M\in \mathbb R^{m\times (d+1)}$
 average under the Gaussian probability distribution.

It follows that in this case
Algorithm \ref{algapprls} can fail only  
for a narrow class of pairs $F$ and $M$ where
$F$ denotes orthogonal matrices in $\mathbb R^{k\times m}$
 and $M$ denotes matrices in $\mathbb R^{m\times (d+1)}$,
and even in the unlikely case of failure we can still have good chances for success 
if we apply 
heuristic recipes of our Section \ref{smngflr}.

\medskip

\medskip


{\bf {\LARGE {Appendix} }}
\appendix 




\section{Randomized Matrix Computations}\label{srmc}


  

\begin{theorem}\label{thrnd} 
Suppose that $A$ is an  
$m\times n$ matrix of full rank $k=\min\{m,n\}$, 
$F$ and $H$ are $r\times m$ and 
$n\times r$  matrices, respectively, for $r\le k$,
and  the  entries of these two matrices are nonconstant linear combinations
of finitely many i.i.d. random variables $v_1,\dots,v_h$. 

Then 
the matrices $F$, $FA$, $H$, and $AH$  
have full rank $r$ 

(i) with probability 1
if $v_1,\dots,v_h$ are Gaussian variables
and 

(ii) with a probability at least $1-r/|\mathcal S|$
if they are random variables sampled under the
uniform probability distribution from 
a finite set $\mathcal S$ having cardinality 
$|\mathcal S|$.  
\end{theorem}


\begin{proof}
The determinant, $\det(B)$,
of any $r\times r$ block $B$ of a matrix 
 $F$, $FA$, $H$, or $AH$
is a polynomial of degree $r$ in the variables  $v_1,\dots,v_h$,
and so the equation $\det(B)=0$ 
  defines an algebraic variety of a lower
dimension in the linear space of these variables
(cf. \cite[Proposition 1]{BV88}). 
Clearly, such a variety has Lebesgue  and
Gaussian measures 0, both being absolutely continuous
with respect to one another. This implies part (i) of the theorem.
Derivation of part (ii) from a celebrated lemma of \cite{DL78},
 also known from \cite{Z79} and \cite{S80},  
is a well-known pattern, specified in some detail
 in \cite{PW08}.
\end{proof}


\begin{lemma}\label{lepr3} ({\rm Rotational invariance of a Gaussian matrix.})
Suppose that $k$, $m$, and $n$  are three  positive integers,
$G$ is an 
 $m\times n$  Gaussian matrix, and
$S$ and $T$ are $k\times m$ and 
$n\times k$
 orthogonal matrices, respectively.

Then $SG$ and $GT$ are Gaussian matrices.
\end{lemma}


We state the following estimates for real matrices,
but similar estimates in the case of complex  
 matrices can be found in \cite{D88}, \cite{E88}, \cite{CD05}, 
and \cite{ES05}:

\begin{definition}\label{defnrm} {\rm Norms of random matrices and
expected value of 
a random variable.}
Write 
$\nu_{m,n}=||G||$,
$\nu_{m,n}^+=||G^+||$,  and
$\nu_{m,n,F}^+=||G^+||_F$,
for  a  Gaussian $m\times n$ matrix  $G$,
and write $\mathbb E(v)$ for the expected value of 
a random variable $v$.
($\nu_{m,n}=\nu_{n,m}$,
$\nu_{m,n}^+=\nu_{n,m}^+$,
and $\nu_{F,m,n}=\nu_{F,n,m}$,
for all pairs of $m$ and $n$.) 
\end{definition}


\begin{theorem}\label{thsignorm}
(Cf. \cite[Theorem II.7]{DS01}.)
Suppose 
that $m$ and $n$ are positive integers,
$h=\max\{m,n\}$, $t\ge 0$.
Then 

(i) {\rm Probability}$\{\nu_{m,n}>t+\sqrt m+\sqrt n\}\le
\exp(-t^2/2)$ and 

(ii) $\mathbb E(\nu_{m,n})< 1+\sqrt m+\sqrt n$.
\end{theorem}


\begin{theorem}\label{thsiguna} 
Let $\Gamma(x)=
\int_0^{\infty}\exp(-t)t^{x-1}dt$
denote the Gamma function and let  $x>0$. 
Then 

(i)  {\rm Probability} $\{\nu_{m,n}^+\ge m/x^2\}<\frac{x^{m-n+1}}{\Gamma(m-n+2)}$
for $m\ge n\ge 2$,

(ii) {\rm Probability} $\{\nu_{n,n}^+\ge x\}\le 2.35 {\sqrt n}/x$ 
for $n\ge 2$,

(iii)  $\mathbb E(\nu^+_{m,n})\le e\sqrt{m}/|m-n|$,
provided that $m\neq n$ and $e=2.71828\dots$.
\end{theorem}


\begin{proof}
 See \cite[Proof of Lemma 4.1]{CD05} for part (i), 
\cite[Theorem 3.3]{SST06} for part (ii),  and
\cite[Proposition 10.2]{HMT11} for part (iii).
\end{proof}

 
The  probabilistic upper bounds of Theorem \ref{thsiguna}
on $\nu^+_{m,n}$ are reasonable already  
in the case of square matrices, that is, where $m=n$,
but are strengthened very fast as the difference $|m-n|$ grows from 1.


Theorems \ref{thsignorm} and \ref{thsiguna}
combined imply that an $m\times n$ Gaussian
matrix is well-conditioned  
unless the integer $m+n$ is large or the integer $|m-n|$ is close to 0.
With some grain of salt
we can still consider  such a matrix  
well-conditioned  
 even
  where the integer $|m-n|$ is small or  vanishes
provided that  the integer $m$ is not large.
Clearly, these properties can be extended immediately to all submatrices.


\medskip


\noindent {\bf Acknowledgements:}
Our research has been supported by NSF Grant CCF 1116736 
and PSC CUNY Award  68862--00 46.
We are also grateful to Vsevolod Oparin
for a pointer to relevant bibliography.



\end{document}